\documentclass[arxiv,preprint,usenames,dvipsnames]{imsart} 
\RequirePackage[OT1]{fontenc}
\RequirePackage{amsthm,amsmath,amssymb}
\RequirePackage[numbers]{natbib}
\RequirePackage[colorlinks]{hyperref}
\usepackage{caption,subcaption} 
\usepackage{mathrsfs,enumerate,bbm,latexsym,graphicx}
\usepackage{hyperref}
\usepackage{ulem}

\usepackage[letterpaper,margin=3cm]{geometry}

\startlocaldefs
\numberwithin{equation}{section}
\theoremstyle{plain}
\newtheorem{thm}{Theorem}[section]

\usepackage{upgreek}

\usepackage{tikz}

\renewcommand{\Re}{\mathop{\mathrm{Re}}}

\newtheorem{corollary}[thm]{Corollary}

\newtheorem{lemma}[thm]{Lemma}

\newtheorem{proposition}[thm]{Proposition}
\newtheorem{remark}[thm]{Remark}
\newtheorem{theorem}[thm]{Theorem}

\newtheorem{question}{Question}
\newtheorem{assumption}{Assumption}

\usepackage{xcolor} 

\newcommand{\df}[1]{\textbf{#1}}

\def\Z{\mathbb{Z}}

\def\S{\mathcal{S}}
\def\C{\mathcal{C}}

\newcommand{\prob}[1]{\mathbb{P}\left(#1\right)}
\newcommand{\probsub}[2]{\mathbb{P}_{#2}\left(#1\right)}

\newcommand{\abs}[1]{\left|#1\right|}
\newcommand{\norm}[1]{\left\|#1\right\|}

\newcommand{\ceil}[1]{\left\lceil#1\right\rceil}

\newcommand{\eps}{\epsilon}

\endlocaldefs

\begin{document}

\begin{frontmatter}
\title{Coarsening model on $\Z^{\lowercase{d}}$ with biased zero-energy flips and an exponential large deviation bound for ASEP}
\runtitle{Coarsening with biased flips}

\begin{aug}
\author{\fnms{Michael} \snm{Damron}},
\author{\fnms{Leonid} \snm{Petrov}} \and
\author{\fnms{David} \snm{Sivakoff}}
\end{aug}
\begin{abstract}
We study the coarsening model (zero-temperature Ising Glauber dynamics) on $\Z^d$ (for $d \geq 2$) with an asymmetric tie-breaking rule. This is a Markov process on the state space $\{-1,+1\}^{\mathbb{Z}^d}$ of ``spin configurations'' in which each vertex updates its spin to agree with a majority of its neighbors at the arrival times of a Poisson process. If a vertex has equally many $+1$ and $-1$ neighbors, then it updates its spin value to $+1$ with probability $q \in [0,1]$ and to $-1$ with probability $1-q$. 
The initial state of this Markov chain 
is distributed according to a product measure with probability $p$ for a spin to be $+1$.
In this paper, we show that for any given $p>0$, there exist $q$ close enough to 1 such that a.s. every spin has a limit of $+1$. This is of particular interest for small values of $p$, for which it is known that if $q=1/2$, a.s. all spins have a limit of $-1$. For dimension $d=2$, we also obtain near-exponential convergence rates for $q$ sufficiently large, and for general $d$, we obtain stretched exponential rates independent of $d$. Two important ingredients in our proofs are refinements of block arguments of Fontes-Schonmann-Sidoravicius and a novel exponential large deviation bound for the Asymmetric Simple Exclusion Process.
\end{abstract}



\end{frontmatter}

\section{Introduction}

The coarsening model on $\Z^d$ with nearest-neighbor edges is defined as follows.   Let $\S = \{-1,1\}^{\Z^d}$.  Each vertex, $x\in\Z^d$, has associated with it an independent Poisson clock of rate 1 and a spin $\sigma_x^t \in \{-1,+1\}$.  The state of the system at time $t$ is then $\sigma^t = \{\sigma_x^t\}_{x\in \Z^d} \in \S$. Define the energy at $x$ to be $e_x^t = -\sum_{y\sim x} \sigma_x^t \sigma_y^t$, where $y\sim x$ means that $y$ is a neighbor of $x$.  When the clock at $x$ rings, say for the $i^\text{th}$ time at time $t$, $x$ updates its spin as
$$
\sigma_x^t = \begin{cases}
\sigma_x^{t-} & \hbox{if $e_x^{t-} < 0$} \\
-\sigma_x^{t-} & \hbox{if $e_x^{t-} > 0$} \\
\xi^i_x & \hbox{if $e_x^{t-} = 0$},
\end{cases}
$$
where $\xi_x^i$ is an independent (of everything else) random variable with 
\[
\prob{\xi^i_x = +1} = 1-\prob{\xi^i_x=-1}=q \in [0,1].
\]
(Formally, we assign an i.i.d. sequence $(\xi_x^i)_{i\in \mathbb{N}}$ of variables to each site, and use the $i$-th variable at site $x$ at the time of the $i$-th update of $x$.) The variables $\xi^i_x$ break ties: when $e_x^{t-}=0$, there are equal numbers of $+1$ and $-1$ neighbors of $x$, so instead of assuming the majority spin, $\sigma_x^t$  assumes an independent spin $\xi_x^i$.

The initial state, $\{\sigma_x^0\}_{x\in \Z^d}$ is assumed to be drawn from the product measure with probability $p$ for $+1$.  Let $\mathbb{P}^\sigma_{q}(\cdot)$ denote the law of the process with initial configuration $\sigma^0=\sigma$, and denote by $\prob{\cdot} = \probsub{\cdot}{p,q}$ the joint law of the process and the initial state $\sigma$ drawn from the product measure.


\subsection{Main result}

Observe that for any $q$, the dynamics are attractive with respect to the standard partial ordering of states, where $\sigma \leq \hat \sigma$ iff $\sigma_x \leq \hat\sigma_x$ for all $x \in \Z^d$.  This implies that for $q<\hat q$, if $\sigma \leq \hat \sigma$, then one can couple $\mathbb{P}^\sigma_{q}$ and $\mathbb{P}^{\hat\sigma}_{\hat q}$ such that $\sigma^t \leq \hat \sigma^t$ for all $t$, simply by coupling $\xi^i$ and $\hat\xi^i$ such that $\xi^i \leq \hat\xi^i$ and using the same clocks for both processes. We then define
\[
q_c = q_c(p,d) = \inf\left\{q \in [0,1] : \mathbb{P}_{p,q}\left(\lim_{t\to\infty} \sigma^t_0 = +1\right) = 1\right\},
\]
whenever the set on the right is nonempty. The main purpose of this paper is to ask the following question and to give a partial result.
\begin{question}
Is $q_c = q_c(p,d)$ strictly between $\frac12$ and $1$ for some $p\in (0,1)$ and $d\ge 2$?
\end{question}
We also define
$$
p_c = p_c(q,d) = \sup\left\{p \in [0,1] : \mathbb{P}_{p,q}\left(\lim_{t\to\infty} \sigma^t_0 = -1\right) = 1\right\}.
$$
It was shown in~\cite{FSS} that $p_c(1/2,d) > 0$ for all $d\ge 2$, and in~\cite{Morris} it was shown that $p_c(1/2,d) \to 1/2$ as $d\to\infty$. Therefore, one might think that $q_c(p,d) = 1$ for some small enough $p>0$ and large enough~$d$. In other words, if $p>0$ is small and fixed, and we choose $d$ large enough so that $p_c(1/2,d)>p$, changing $q$ from $1/2$ (where the system converges to $-1$) to $q \in (1/2,1)$ might not be enough to drive the system to $+1$. Our main result shows this is false.
\begin{thm}\label{thm: main_thm}
For any fixed $p \in (0,1)$ and $d \geq 1$, one has $q_c(p,d) < 1$.
\end{thm}

\begin{remark}
A simple consequence of Theorem~\ref{thm: main_thm} is that the set defining $q_c$ is nonempty. In contrast, it is important to note that the analogue of Theorem~\ref{thm: main_thm} cannot hold on $k$-regular trees with even values of $k \geq 4$. (If $k$ is odd, there are no ties, and so $q$ has no effect.) Indeed, one can show that for $p$ small enough, even when $q=1$, one has $\sigma_x^t = -1$ for all $t$ with positive probability. As a result, the set defining $q_c(p)$ for such $p$ is empty. One can prove this by dominating our coarsening model by a $-1\to+1$ bootstrap percolation process with threshold $\frac{k}{2} +1$. Therefore if $p$ is smaller than the critical probability for bootstrap percolation (which is positive by \cite{BPP,CRL}), some spins in the coarsening model will stay $-1$ forever.
\end{remark}

Theorem~\ref{thm: main_thm} is a direct consequence of more precise bounds that we derive on fixation times. 
\begin{thm}\label{fixation times}
Let $p>0$.
\begin{enumerate}
\item For $d=2$, there exists $q<1$ and $C>0$ such that
\[
\mathbb{P}_{p,q}(\sigma_0^s = -1 \text{ for some }s \geq t) \leq \exp\left(-C\frac{t}{\log^2 t}\right) \text{ for all large } t > 0.
\]
\item For $d\ge 3$, there exists $q<1$ such that for any real $\beta>\min(d-1,3)$ 
\[
\mathbb{P}_{p,q}(\sigma_0^s = -1 \text{ for some }s \geq t) \leq \exp\left( -t^{1/\beta} \right) \text{ for all large } t > 0.
\]
\end{enumerate}
\end{thm}

Theorem~\ref{fixation times} establishes a near-exponential fixation time in dimension 2 due to an ASEP (Asymmetric Simple Exclusion Process) large deviation estimate, which may be of independent interest, and which we describe below. For dimension $d\ge 3$, we establish stretched-exponential bounds with dimension-independent exponents by applying erosion time estimates of~\cite{CMST,Lac}.

To formulate the ASEP large deviation bound, let us first recall the definition of the process. 
The ASEP is a continuous time Markov chain on particle configurations 
$\mathsf{x}=(\mathsf{x}_1>\mathsf{x}_2>\ldots )$ in $\mathbb{Z}$
(each location can be occupied by at most one particle).
For our purposes, it suffices to consider configurations which have a rightmost particle.
Each particle has an independent clock with exponential waiting time of mean $1$.
When the clock rings, the particle jumps to the right with probability $q$ 
or to the left with probability $1-q$, provided that the destination is unoccupied
(otherwise the jump is forbidden).
Let us denote $\upgamma:=2q-1$.
We consider the ASEP started from the \textit{step} initial configuration
$\mathsf{x}_j(0)=-j$, $j=1,2,\ldots $,
and will denote the corresponding probability measure by $\mathbb{P}_{\mathrm{step},q}$.

\begin{thm}\label{thm:ASEP_theorem_intro}
	Let $q>\frac{1}{2}$, i.e., the ASEP has drift to the right. 
	Fix any
	$\varepsilon \in (0,1)$
	and set $m=\lfloor \frac{t}{4}(1-\varepsilon) \rfloor $.
	There exists $\mathsf{C}>0$ such that for all $t$ large enough we have
	\begin{equation}\label{ASEP_large_dev_bound_intro}
		\mathbb{P}_{\mathrm{step},q}\bigl(\mathsf{x}_m(t/\upgamma)<0\bigr)\le \mathsf{C} e^{-t \Phi_+(\varepsilon)},
	\end{equation}
	where $\Phi_+(\varepsilon)$ is an explicit function given in \eqref{I_epsilon} below.
	It is positive and increasing for $\varepsilon>0$ 
	and behaves as $\Phi_+(\varepsilon)\sim\frac{2}{3}\varepsilon^{\frac{3}{2}}$ as $\varepsilon\to0+$.
\end{thm}
Theorem~\ref{thm:ASEP_theorem_intro} is a one-sided large deviation bound 
for the integrated ASEP current through zero 
$\mathfrak{h}_0(t):=\#\{\textnormal{particles to the right of zero at time $t$}\}$.
Indeed, it is known \cite[Theorem 5.12]{Liggett_IPS} 
that the current satisfies
the following strong law of large numbers:
\begin{equation*}
	t^{-1}\mathfrak{h}_0(t/\upgamma)
	\to
	1/4\quad\textnormal{almost surely, as }
	 t\to\infty.
\end{equation*}
(%
	Moreover,
	the fluctuations of $\mathfrak{h}_0$ around $t/4$ have order $t^{1/3}$
	and are governed by the GUE Tracy--Widom distribution
	\cite{TW_asymptotics}.%
)
The probability in the left-hand side of \eqref{ASEP_large_dev_bound_intro} 
is essentially the same as
\begin{equation*}
	\mathbb{P}_{\mathrm{step},q}
	\left(
		\frac{\mathfrak{h}_0(t/\upgamma)}{t}
		<
		\frac{1-\varepsilon}{4}
	\right),
\end{equation*}
and we obtain a one-sided exponential bound for it 
(see also Remark~\ref{rmk:ASEP_more_discussion}
below for further background on this bound).
The proof of Theorem \ref{thm:ASEP_theorem_intro} is based on asymptotic
analysis of the Fredholm determinantal representation for probability
distributions in ASEP with the step initial condition and is given in Section~\ref{section3}.
Similar analysis was employed in \cite{TW_asymptotics} to obtain GUE
Tracy--Widom fluctuation behavior for the ASEP, putting this process
into the so-called Kardar--Parisi--Zhang universality class.

\subsection{Background}
We now discuss past results as they relate to question 1. For $d=1$, when $p \in (0,1)$ and $q = 1/2$, almost surely, $\sigma_0^t$ does not have a limit \cite{Arr}. Therefore $q_c(p,1) \geq 1/2$ for all $p \in (0,1)$. However it is not difficult to show that for $d=1$ and $p \in (0,1)$, whenever $q > 1/2$, one has $\mathbb{P}_{p,q}(\lim_t \sigma_0^t=+1)=1$. Using symmetry, we conclude that $q_c(p,1) = 1/2$ for all $p \in (0,1)$.

For $d=2$, it is known \cite[Theorem~2]{NNS}  that when $p=q=1/2$, one has
\[
\mathbb{P}_{1/2,1/2}\left(\lim_{t\to\infty} \sigma_0^t \text{ does not exist}\right)=1.
\]
(This is also believed to hold for sufficiently low dimensions \cite{OKR}, while it is thought that fixation may occur when $p=q=1/2$ and $d$ is sufficiently large~\cite{SKR}.) Therefore $q_c(1/2,2) \geq 1/2$. By monotonicity of the dynamics in $p$, one furthermore has 
\[
q_c(p,2) \geq 1/2 \text{ for all }p \leq 1/2. 
\]

In general dimensions $d \geq 2$, \cite{FSS} showed that if $p$ is close to 1 and $q=1/2$, then $\sigma_0^t$ converges to $+1$ almost surely. By symmetry between $+ 1$ and $- 1$, then, if $p$ is close to 0 and $q=1/2$, then $\sigma_0^t$ converges to $-1$ almost surely. Thus we deduce that
\[
\text{for }d \geq 2,~ q_c(p,d) \begin{cases}
\leq 1/2 & \text{ for }p \text{ close to }1 \\
\geq 1/2 & \text{ for }p \text{ close to }0
\end{cases}.
\]
As a consequence of this and symmetry, one strategy to prove $q_c(1/2,d)=1/2$ for some $d \geq 2$ would be to show continuity of $q_c(p,d)$ in $p$.

{In dimension $d=2$, \cite{Lac2} considered the same dynamics as we do here, and studied the asymptotic shape of a large region of $-1$'s surrounded entirely by $+1$'s. In~\cite{Lac2}, $h>0$ represents an external magnetic field, and $q$ and $h$ are related by $q = e^{h}/(2\cosh(h))$. In the case $q=1$ ($h=\infty$), \cite{Lac2} showed that the asymptotic shape of an (initial) $L$ by $L$ square of $-1$'s satisfies a Law of Large Numbers, in the sense that the $-1$ region, when rescaled by $L$ with time sped up by $L$, follows a deterministic evolution that shrinks to a point in finite time. Moreover, \cite{Lac2} remarks that a similar result holds for all $q>1/2$ ($h>0$) and all suitable regions of $-1$ spins. The case $q=1$, when started from $-1$ in the first quadrant and $+1$ elsewhere, corresponds to the TASEP (Totally Asymmetric Simple Exclusion Process) started from the step initial condition, for which~\cite{Rost} computed the almost-sure limiting particle density (shape) when space and time are scaled linearly.
These results hint at the near-exponential fixation time in part 1 of Theorem~\ref{fixation times}, but are insufficient to derive it because we require an exponential probability bound on the (linear) speed at which $-1$ regions shrink.}

To give a positive answer to question 1, it would suffice to show that there are some (probably small) values of $p$ such that if $q > 1/2$ is arbitrarily close to $1/2$, then the system will fixate to $-1$. The difficulty is that if $p$ is small enough, then for long periods of time (depending on $q$), the system will behave as if $q=1/2$, and thus will want to fixate to $-1$ (due to \cite{FSS}). Showing that the system will not then ``change directions'' at a later time and fixate to $+1$ involves analyzing the configuration $(\sigma_x^t : x \in \mathbb{Z}^d)$ at a large $t$, when the variables are highly correlated. Unfortunately, there are few tools available for such analysis.

It is worth noting that other tie-breaking rules have been used in the literature. One option is to set $\sigma_x^t = \sigma_x^{t-}$ when $e_x^{t-}=0$. This rule is considered, for example, in \cite{BCOTT} (for the discrete-time analogue of the coarsening model, usually called the majority vote model) and the process has the same behavior as ours on $\mathbb{Z}^d$ when an additional edge is placed between each vertex and itself. On this new graph, a vertex has $2d+1$ neighbors, so there are no ties, and the additional edge keeps $\sigma_x$ from flipping when $x$'s original $2d$ neighbors have an equal number of $+1$ and $-1$ spins. Here, one can apply a result of \cite{NNS}, which applies to certain odd-degree graphs to deduce that for each $x$, $\lim_{t\to\infty} \sigma_x^t$ exists almost surely. However for any $p>0$, there exist vertices that fixate to $-1$, since any side-length two cube of initially $-1$ spins is stable for all time.

Last, we mention that there exist graphs like finite width slabs (graphs of the form $\mathbb{Z}^d \times \{0, \ldots, k\}$ for any $k \geq 3$) which have vertices of even degree but for which $q_c(p)$ cannot be strictly less than $1$ for $p< 1$. In these graphs, one can construct finite sets of initially $-1$ spins that are stable for all time.


\subsection{Sketch of proof}\label{sec: sketch}

Due to the above discussion, our main question has to do with the balance between low values of $p$ driving the system toward $-1$ and a bias $q>1/2$ driving the system toward $+1$. Based on the results of \cite{FSS}, the first effect occurs on timescales that are stretched exponential: for small $p$, one has
\[
\mathbb{P}_{p,1/2} (\sigma_0^t = +1) \leq C_1e^{-C_2 t^\alpha},
\]
where $\alpha < 1$ is a function of $d$. On the other hand, it is reasonable to believe that the second effect, due to $q > 1/2$, takes places on exponential time-scales. (At least in $d=2$ case this follows from a comparison to ASEP.) 

For large $q$, however, the bias has a strong effect, even when $p$ is small, and allows us to prove Theorem~\ref{thm: main_thm}. We now sketch the argument. First consider $q=1$. In this case, whenever a vertex has $d$ or more neighbors with $+1$ spin, it flips to $+1$. We can therefore compare to a Modified Bootstrap Percolation (MBP) process which is defined as follows. Each site begins with a $\pm 1$ spin, and the distribution of these spins is i.i.d. with probability $p>0$ to be $+1$. At each time $t=1, 2, \ldots$, each vertex with at least $d$ neighbors with spin $+1$, all in distinct directions ($\pm e_i$ for $i=1, \ldots, d$) flips to $+1$. All other spins remain the same --- see the definition in Section~\ref{Fixation for q=1}. It is known \cite{schonmann} that for this process, almost surely, each spin eventually fixates to $+1$.  (It is not sufficient to consider the standard Bootstrap Percolation process, which requires only $d$ neighbors with spin $+1$ to flip to $+1$, with no restriction on the directions being distinct. This is due to our identifying $2^d$-sized blocks with sites in the bootstrap percolation process in the proof of Proposition~\ref{q=1 thm}, so blocks on opposite sides of a given block have no neighboring vertices in common, and therefore do not aid in the growth of $+1$'s.) 

In the original coarsening process with $q=1$, a $+1$ spin can flip to $-1$ if it has at least $d+1$ neighbors with $-1$ spins, so the coarsening dynamics are not exactly the same as those of the MBP process. However, a comparison with MBP shows that sufficiently large squares $\Lambda$ have the following property. With high probability (in the size of the square), there is a $t$ such that all vertices in $\Lambda$ have spin $+1$ at time $t$. Once all the spins in a square become $+1$, they will remain $+1$ forever. We then conclude that for $q=1$, the coarsening model fixates to $+1$. This is stated in Proposition~\ref{q=1 thm}. 

To allow $q<1$, we choose a large square $\Lambda$ and pick $q<1$ so that with probability close to 1, at a fixed large time $t$, all spins in $\Lambda$ are $+1$. This places us in a variant of the setting of the FSS argument \cite{FSS}, which is designed to show that the coarsening model will fixate to $+1$ if the initial condition is sufficiently biased to $+1$. 
In Theorem~\ref{thm: block_FSS}, we present a version of the FSS argument in which the initial condition is constant on blocks, and which (in the case $d=2$) compares erosion of blocks of $-1$ spins to the behavior of the ASEP particle system rather than the SEP, as was done in \cite{FSS}. This comparison allows for a faster fixation rate ($\exp(-Ct/\log^2 t)$) for $d=2$ than was given in \cite{FSS} ($\exp(-Ct^{1/2 - \epsilon})$). To do this, we give a large deviation bound for ASEP (Theorem~\ref{thm:ASEP_theorem_intro}), and this may be of independent interest. In higher dimensions, we apply the results of \cite{CMST,Lac} for the speed of unbiased corner growth to obtain a fixation rate of $\exp(-t^{1/\beta})$ for any $\beta>\min\{3,d-1\}$, and this rate also improves on that given in \cite{FSS}. In Section~\ref{sec: finish}, we combine Theorem~\ref{thm: block_FSS} with Proposition~\ref{q=1 thm} to derive Theorem~\ref{thm: main_thm}.

\section{Fixation for $q=1$ and $p>0$} \label{Fixation for q=1}
We intend to show that fixation occurs for any initial density $p>0$ when $q=1$.  The argument relies on a result for the Modified Bootstrap Percolation process. \df{Modified Bootstrap Percolation (MBP)} is a discrete-time, monotone growth process whose state (or configuration) at step $n$ is $\zeta_n\in \{0,1\}^{\Z^d}$.  For each $v\in \Z^d$, we say $v$ is \df{occupied} at step $n$ if $\zeta_n(v) = 1$, and is \df{vacant} otherwise.  Given an initial configuration $\zeta_0$, the deterministic dynamics proceed as follows. A vacant site $v$ at step $n$ becomes occupied at step $n+1$ if and only if 
\begin{equation}\label{mbp-rule}
\#\{i\in\{1,\ldots,d\}\,:\, \text{at least one of $v\pm e_i$ is occupied at step }n\}= d,
\end{equation}
where $e_i$ denotes the $i^\text{th}$ standard basis vector, and we let $\zeta_\infty$ denote the pointwise limit of $\zeta_n$. In words, if $v$ is occupied, it remains occupied forever, and if $v$ is vacant, then it becomes occupied if it sees occupied neighbors in all $d$ distinct basis directions.  
The initial configuration is drawn from a product measure $P_\theta$ on $\{0,1\}^{\mathbb{Z}^d}$ with probability $\theta$ for $1$. For a set $A\subseteq \Z^d$, we say that the initial configuration $\zeta_0$ \df{spans} $A$ if every vertex in $A$ eventually becomes occupied, so $\zeta_\infty(v) = 1$ for every $v\in A$.  Define the configuration $\zeta_0^A$ as
$$\zeta_0^A(v) = \begin{cases} \zeta_0(v) & v\in A \\ 0 & v\in A^c. \end{cases}$$
We say that $\zeta_0$ \df{internally spans} $A$ if $\zeta_0^A$ spans $A$.

Let $\Lambda_n = [0,n-1]^d$.  We will make use of the following bound on the probability that MBP internally spans the box $\Lambda_n$, which is Proposition 3.2 in~\cite{schonmann}.

\begin{lemma}
\label{bootstrap}
Fix $\theta>0$ and $d\ge 2$.  Then there exists a constant $c>0$ (depending on $\theta$ and $d$) such that $$P_\theta(\zeta_0 \text{ internally spans } \Lambda_n) \geq 1 - e^{- c n}.$$
\end{lemma}

\begin{remark} \label{mbp remark} The definition of the MBP process states that all vertices are updated simultaneously at each step. However, Lemma~\ref{bootstrap} will still hold under other updating rules.  The only property necessary for such an updating rule is that at each step, if there is a vertex that is vacant and can be made occupied, then some vertex is made occupied. That is, the order in which vertices are occupied does not matter, as long as no vertex is deliberately ignored.  This is the case if, for example, vertices attempt to update their states in continuous time according to independent Poisson processes.
\end{remark}

We intend to show that spanning in MBP implies fixation to the all $+1$ state for the coarsening model when $q=1$ and $p>0$.  A key observation in the case of $q=1$ is that a block of vertices initialized at $+1$ will remain $+1$ forever. Indeed, if all $x \in \Lambda_n$ (for $n \geq 2$) have $\sigma_x^0=+1$, then whenever any vertex attempts to flip, it has at least $d$ neighbors in the $+1$ state, and therefore will not flip. The precise statement follows.

\begin{lemma}
\label{cubes fixate}
If $q=1$, $n\geq 2$ and $\sigma^0_x= +1$ for all $x\in \Lambda_n$, then $\sigma^t_x = +1$ for all $x\in\Lambda_n$ and $t\geq 0$ almost surely.
\end{lemma}

The next theorem says that in $\Z^d$, large boxes tend to fixate to all $+1$, regardless of the initial state outside of the box.

\begin{proposition}
\label{q=1 thm}
If  $p>0$ and $q=1$, then there exists $c>0$ (depending on $p$ and $d$) such that 
\[
\probsub{\lim_{t\to\infty} \sigma^t_x = +1 \text{ for all } x\in \Lambda_n \, \middle|\, \sigma^0_x = -1 \text{ for all } x\in \Lambda_n^c}{p,1} \ge 1-e^{-cn}.
\]
\end{proposition}
\begin{proof}
When $q=1$, $\lim_{t\to\infty} \sigma^t_x$ exists almost surely for each $x\in \Z^d$ because the number of energy-lowering flips at $x$ is almost surely finite (see the remark after Theorem~3 in \cite{NNS}), and each vertex can undergo at most one energy-neutral flip (to $+1$).

Assume first that $n$ is even.  We identify $\sigma^0$, the initial spin configuration, with $\zeta_0$, an initial MBP configuration.  For each $v\in \Z^d$, we set $\zeta_0(v) = 1$ if $\sigma^0_x=+1$ for every $x\in 2v + \Lambda_2$ and $\zeta_0(v)=0$ otherwise. Then under $\probsub{\cdot}{p,1}$, the initial MBP configuration $\zeta_0\sim P_\theta$ is distributed according to product measure with probability $\theta = p^{2^d}$ for 1.  We claim that if $\zeta_0$ internally spans $\Lambda_{n/2}$, then $\lim_{t\to\infty} \sigma^t_v = +1$ for all $v\in \Lambda_n$.  To see why, we argue by induction on the bootstrap time step $j$. So suppose for a fixed $x\in \Lambda_{n/2}$ that $\zeta_{j}(x)=1$ for some $j\ge 1$ and also that for any $y \in \Lambda_{n/2}$ such that $\zeta_{j-1}(y)=1$, every vertex in $2y+\Lambda_2$ eventually fixates to $+1$ in the Glauber dynamics.  We need to show that every vertex in $2x+\Lambda_2$ eventually fixates to $+1$ in the Glauber dynamics. To start the induction, first note that if $\zeta_0(x)=1$, then for each $v\in 2x+\Lambda_2$, we have $\sigma^0_v = +1$, and by Lemma~\ref{cubes fixate}, $\sigma^t_v = +1$ for all $t \geq 0$.

If $\zeta_{j-1}(x)=1$, then by the induction hypothesis it follows that every vertex in $2x+\Lambda_2$ is eventually in the $+1$ state.  If $\zeta_{j-1}(x)=0$, then $x$ has $d$ neighbors, $y_1, \ldots, y_d\in \Lambda_{n/2}$, in different directions such that $\zeta_{j-1}(y_1)=\cdots=\zeta_{j-1}(y_d)=1$.  By symmetry of the lattice, we may suppose that these $d$ neighbors are $x-e_1, \ldots, x-e_d$.  We will now use induction on the distance from the vertex $2x$ to show that every vertex in $2x+\Lambda_2$ eventually fixates to $+1$ in the Glauber dynamics. Fix $0\le k \le d$, and suppose that for every $y\in \Lambda_2$ such that $\|y\|_1 \le k-1$ (this set is empty if $k=0$), the vertex $2x+y$ eventually fixates in the $+1$ state. Consider a vertex $2x+z$ with $z\in \Lambda_2$ such that $\|z\|_1 = k$; without loss of generality, suppose $z = e_1+\cdots+e_{k}$ (if $k=0$, then $z=0$).  For each $i\in [1,k]$, the vertex $2x+z - e_i$ is eventually in the state $+1$ by the induction hypothesis (on $k$), and for $i\in[k+1,d]$, the vertex $2x+z-e_i = 2(x-e_i) + (z+e_i) \in 2(x-e_i) + \Lambda_2$ is eventually in the $+1$ state, since $\zeta_{j-1}(x-e_i)=1$. Therefore, $2x+z$ eventually has $d$ neighbors frozen in the $+1$ state, and since it attempts to flip at arbitrarily large times, $2x+z$ will eventually fixate to $+1$, which concludes the induction on $k$.  Therefore, we have that the cube $2x+\Lambda_2$ will eventually fixate to $+1$, which finishes the induction on $j$.

Now we have 
\begin{align*}
 &\probsub{\lim_{t\to\infty} \sigma^t_x = +1 \text{ for all } x\in \Lambda_n \, \middle|\, \sigma^0_x = -1 \text{ for all } x\in \Lambda_n^c}{p,1} \\
 &\hspace{3cm} \ge\probsub{\zeta_0 \text{ internally spans } \Lambda_{n/2}}{p,1} \ge 1-e^{-c'n/2}
 \end{align*}
for all even $n$, where $c' = c'(d,p^{2^d})$ is the constant in Lemma~\ref{bootstrap}. 

When $n$ is odd, we can apply the even $n$ result to each of the $2^d$ boxes $x+\Lambda_{n-1}$ for $x\in\{0,1\}^d$. If all of these boxes have initial configurations that eventually flip to all $+1$, then every vertex in the box $\Lambda_n$ eventually flips to $+1$, so this happens with probability at least $1-2^de^{-c'(n-1)/2}$. By choosing $c<c'/2$ small enough, this gives the result for all $n$.
\end{proof}

Proposition~\ref{q=1 thm} will be used later in Section~\ref{sec: finish} to prove that large boxes can be made to have all $+1$ spins with high probability at a large time, conditional on the state outside the box, even when $q$ is $<1$ (but very close to 1). Specifically, the reader should see \eqref{box q*}, which states that for a given $\epsilon>0$ and $p>0$, there exist $L_0, t_0 > 0$ and $q^*<1$ such that
\[
\mathbb{P}_{p,q^*}\left( \sigma_x^{t_0}=+1 \text{ for all } x \in \Lambda_{L_0} \mid \sigma_x^0 = -1 \text{ for all } x \in \Lambda_{L_0}^c\right) > 1-\epsilon.
\]
This estimate will give us the initial scale $L_0$ at which we will apply Theorem~\ref{thm: block_FSS}.

\section{Decay of boxes}\label{section3}
Let $T$ be the time for the configuration in the $L^d$ rectangle, $\Lambda_L$, to reach all $+1$, when the dynamics (with $q$-biased tie breaking) is run with an initial configuration of all $-1$ inside $\Lambda_L$ and all $+1$ outside $\Lambda_L$. It was proved {by~\cite{FSS} for dimension $d=2$, \cite{CMST} for dimension $d=3$, and by~\cite{Lac} for dimension $d\ge 4$ that if $q=1/2$, then $T$ is at most order $L^2$ (up to logarithmic corrections) with high probability.} The precise result is as follows.
\begin{theorem}[{Theorem 1.3 in \cite{FSS}, }Theorem 3.1 in \cite{CMST} and Theorem 2.2 in \cite{Lac}]\label{box decay}
Let $q=1/2$ and {$d\ge 2$.} There exists a constant $c>0$ (not depending on $d$) such that
\[
\probsub{T \ge L^2 (\log L)^c}{} \le \frac{c}{L} \qquad \text{for all $L$.}
\]
\end{theorem}
In fact, \cite{FSS} proved an exponential probability bound, without the logarithmic correction for $d=2$. Note that by monotonicity, Theorem~\ref{box decay} also holds when $q\ge 1/2$. However, in Assumption~\ref{assumption} below, we will require an exponential probability bound (not a polynomial one) on the erosion time of an $L^d$ box. In dimension $d=2$, we are able to prove that $T$ grows linearly with $L$ with an exponential probability bound, and in dimensions $d\ge 3$, we apply Theorem~\ref{box decay} in a straightforward manner to obtain a (nearly) cubic bound on $T$.
\begin{theorem}\label{exponential box decay}
Suppose $q>1/2$. There is a constant $C>0$ such that the following statements hold.
\begin{enumerate}
\item If $d=2$, then
\[
\probsub{T \ge CL}{q} \le e^{-L/C} \qquad \text{for all $L\ge 1$.}
\]
\item If $d=3$, then 
\[
\probsub{T \ge CL^2}{q} \le e^{-L/{C}} \qquad \text{for all $L\ge 1$.}
\]
\item If $d\ge 4$, 
\[
\probsub{T \ge L^3 (\log L)^{C}}{q} \le e^{-L/{C}} \qquad \text{for all $L\ge 3$.}
\]
\end{enumerate}
\end{theorem}

\begin{remark}
We suspect that Theorem~\ref{exponential box decay} is essentially sharp (with suboptimal constant $C$) only in case 1 ($d=2$). In general, we expect the bound in case 1 to hold for all $d\ge 2$.
\end{remark}

\begin{proof}
Part $1$ is a direct consequence of attractiveness and Corollary~\ref{cor:sigma_ASEP_bound} below. Part $2$ follows from part $1$ by subdividing the three-dimensional box $[0,L-1]^3$ into $L$ two-dimensional slices, $\{i\}\times [0,L-1]^2$ for $i=0,\ldots, L-1$. Consider the slowed-down dynamics in which the slices must decay in lexicographic order. That is, the vertices in $\{i\}\times [0,L-1]^2$ are not allowed to flip until all of the vertices in $\{i-1\}\times [0,L-1]^2$ are in the $+1$ state. By attractiveness, the original dynamics dominate these slowed-down dynamics. If $T_i$ is the time for the $i^\text{th}$ layer to decay (after layer $i-1$ has decayed) in these slowed-down dynamics, and $C'$ is the constant from part $1$, which we may assume is larger than $1/2$ without loss of generality, then by part $1$,
$$
\probsub{T \ge C'L^2}{} \le \sum_{i=0}^{L-1} \probsub{T_i \ge C'L}{} \le L e^{-L/C'} \le e^{-L/(2C')}
$$
for all $L\ge 4 (2C')^2$. This shows the statement of part 2 with $C = 2C'$ for such $L$. To handle $L \leq 4(2C')^2$, we simply increase $C$.

We now prove part $3$ of the theorem using a restarting argument. Suppose $d\ge 4$, let $c>0$ be the constant from Theorem~\ref{box decay}, and for $k\ge 0$, let $t_k = k L^2 (\log L)^c$. Initialize the configuration $\sigma^0$ to be all $-1$ inside $\Lambda_L$ and all $+1$ outside at time $t=0$. Now define a new Markov process $(\tilde\sigma^t)_{t\ge0}$ as follows. In the time intervals $t\in (t_{k-1},t_k)$, we let $\tilde\sigma^t$ follow the same rules as $\sigma^t$.  At the times $t\in \{t_k : k\ge 1\}$, if there exists $x\in \Lambda_L$ such that $\tilde\sigma_x^{t-} = -1$, then set $\tilde\sigma_y^t = -1$ for all $y\in \Lambda_L$. By the obvious coupling, we can construct $\tilde\sigma^t$ on the same probability space as $\sigma^t$ such that $\tilde\sigma^t_x \leq \sigma^t_x$ for all $t\ge 0$ and $x\in \Z^d$. Since the all $+1$ state is absorbing for $\tilde\sigma^t$, we now have
\begin{equation}
\begin{aligned}
\prob{T \ge L^3 (\log L)^c} &\le \prob{\tilde\sigma^{t_L}_x = -1 \text{ for some $x\in\Lambda_L$}}\\
&= \prod_{k=1}^L \prob{\tilde\sigma^{t_k-}_x = -1 \text{ for some $x\in\Lambda_L$} \ \middle|\ \tilde\sigma^{t_{k-1}}_y = -1 \text{ for all $y\in \Lambda_L$}}\\
& = \left[\prob{T\ge L^2 (\log L)^c}\right]^L\\
&\leq (c/L)^L = \exp[- L\log (L/c)],
\end{aligned}
\end{equation}
which proves part $3$ for large $L$ if $C\ge c$. By choosing $C$ large enough, the statement holds for all~$L\ge 3$.
\end{proof}

The remainder of this section is devoted to proving the ASEP bound of Theorem~\ref{thm:ASEP_theorem_intro}
which, along with attractiveness, implies the following corollary for the $d=2$ coarsening model:
\begin{corollary}
	\label{cor:sigma_ASEP_bound}
	Consider the two-dimensional coarsening model $\{\sigma^t_x\}$, $x=(x^1,x^2)\in\mathbb{Z}^2$,
	with the initial condition 
	$\sigma^0_x=-1$ in $\{(x^1,x^2): x^1\ge 1, x^2\ge 1\}$ and $\sigma^0_x=+1$ elsewhere, and
	$q>1/2$. Fix any $\varepsilon\in(0,1)$.
	There exists $\mathsf{C}>0$ such that for all sufficiently large $t>0$ we have
	\begin{equation}\label{coarsening_2d_estimate}
		\mathbb{P}\left( \textnormal{$\sigma^{t/(2q-1)}_x=+1$ for all $x=(x^1,x^2)$ with 
		$1\le x^{i}\le \lfloor \tfrac{t}{4}(1-\varepsilon) \rfloor$
for $i=1,2$} \right)\ge 1-\mathsf{C} e^{-t \Phi_+(\varepsilon)},
	\end{equation}
	where $\Phi_+(\varepsilon)$ is given by \eqref{I_epsilon} below.
\end{corollary}
\begin{proof}[Proof of Corollary~\ref{cor:sigma_ASEP_bound} modulo Theorem~\ref{thm:ASEP_theorem_intro}]
	The coarsening model with the quadrant initial configuration described in the hypothesis
	is the same as the corner growth and decay model
	with exponential waiting times. 
	The latter model can be coupled to the ASEP with the step
	initial configuration described before 
	Theorem~\ref{thm:ASEP_theorem_intro}, such that
	\begin{equation}\label{ASEP_coupling}
		\mathbb{P}\bigl( \sigma^t_{(m,\ell)}=-1 \bigr)=
		\mathbb{P}_{\mathrm{step},q}\left( \mathsf{x}_\ell(t)<m-\ell \right),\qquad m,\ell \ge1,
	\end{equation}
	see Figure~\ref{fig:ASEP_coupling}.
	Here 
	under the ASEP $\left\{ \mathsf{x}_m(t) \right\}$
	the probabilities of right and left jumps 
	are $q$ and $1-q$, respectively.
	Applying Theorem~\ref{thm:ASEP_theorem_intro} we get the desired estimate
	\eqref{coarsening_2d_estimate}.
	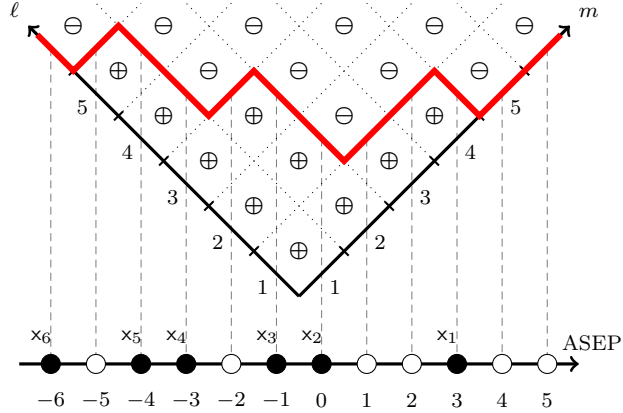
\begin{figure}[htpb]
		\centering
		\begin{tikzpicture}
			[scale=.6]
			\draw[->,very thick] (0,0)--++(6,6) node[above right] {$m$};
			\draw[->,very thick] (0,0)--++(-6,6) node[above left] {$\ell$};
			\draw[->,very thick] (-6.2,-1.5)--++(12.4,0) node[above,yshift=2,xshift=5] {ASEP};
			\foreach \ii in {1,...,5}
			{
				\draw[thick] (\ii+.1,\ii-.1)--++(-.2,.2);
				\draw[thick] (-\ii-.1,\ii-.1)--++(.2,.2);
				\draw[dotted] (\ii,\ii)--++(\ii-6.5,6.5-\ii);
				\draw[dotted] (-\ii,\ii)--++(6.5-\ii,6.5-\ii);
				\node at (\ii-.2,\ii-.8) {$\ii$};
				\node at (-\ii+.2,\ii-.8) {$\ii$};
			}
			\foreach \kk in {(0,1),(1,2),(-1,2),(2,3),(0,3),(-2,3),(3,4),(-3,4),(-4,5),(-1,4)}
			{
				\node at \kk {$\boldsymbol \oplus$};
			}
			\foreach \kk in {(4,5),(2,5),(5,6),(3,6),(1,6),(-1,6),(-3,6),(-5,6),(-2,5),(0,5)}
			{
				\node at \kk {$\boldsymbol \ominus$};
			}
			\node at (1,4) {$\boldsymbol \ominus$};
			\draw[densely dashed, opacity=0.5] (-5.5,5.5) --++ (0,-7);
			\draw[densely dashed, opacity=0.5] (-4.5,5.5) --++ (0,-7);
			\draw[densely dashed, opacity=0.5] (-3.5,5.5) --++ (0,-7);
			\draw[densely dashed, opacity=0.5] (-2.5,4.5) --++ (0,-6);
			\draw[densely dashed, opacity=0.5] (-1.5,4.5) --++ (0,-6);
			\draw[densely dashed, opacity=0.5] (-0.5,4.5) --++ (0,-6);
			\draw[densely dashed, opacity=0.5] (0.5,3.5) --++ (0,-5);
			\draw[densely dashed, opacity=0.5] (1.5,3.5) --++ (0,-5);
			\draw[densely dashed, opacity=0.5] (2.5,4.5) --++ (0,-6);
			\draw[densely dashed, opacity=0.5] (3.5,4.5) --++ (0,-6);
			\draw[densely dashed, opacity=0.5] (4.5,4.5) --++ (0,-6);
			\draw[densely dashed, opacity=0.5] (5.5,5.5) --++ (0,-7);
			\draw[red,line width=2.3] (-5.8,5.8)--(-5,5)--++(1,1)--++(2,-2)--++(1,1)--++(2,-2)--++(2,2)--++(1,-1)--++(1.8,1.8);
			\foreach \pp in {-6,-4,-3,-1,0,3}
			{
				\draw[fill] (\pp+.5,-1.5) circle(6pt);
			}
			\node at (3.3,-.9) {$\mathsf{x}_1$};
			\node at (0.3,-.9) {$\mathsf{x}_2$};
			\node at (-0.7,-.9) {$\mathsf{x}_3$};
			\node at (-2.7,-.9) {$\mathsf{x}_4$};
			\node at (-3.7,-.9) {$\mathsf{x}_5$};
			\node at (-5.7,-.9) {$\mathsf{x}_6$};
			\foreach \pp in {5,4,2,1,-2,-5}
			{
				\draw[fill=white] (\pp+.5,-1.5) circle(6pt);
			}
			\foreach \xx in {-6,...,5}
			{
				\node at (\xx+.5,-2.3) {$\xx$};
			}
		\end{tikzpicture}
		\caption{%
			Coupling between the coarsening model and the ASEP. Rotate the quadrant by $45^\circ$,
			draw an interface between the ``$+$'' and ``$-$'' 
			states, and place a particle of the ASEP under each part of the interface of 
			slope $-1$. Right and left jumps of particles correspond
			to growth and decay of the interface, respectively.
			This coupling implies \eqref{ASEP_coupling}.%
		}
		\label{fig:ASEP_coupling}
	\end{figure}
\end{proof}


\begin{proof}[Proof of Theorem~\ref{thm:ASEP_theorem_intro}]
	{\bf Step 1. Pre-limit Fredholm determinant.}
	We first recall a formula from \cite[Lemma 4]{TW_asymptotics} for
	$\mathbb{P}_{\mathrm{step},q}(\mathsf{x}_m(t/\upgamma)>0)$ with any fixed
	$m\ge1$, $t>0$ as an integral of a Fredholm determinant (recall that $\upgamma=2q-1$). 
	We need some notation. Let $\uptau:=(1-q)/q\in(0,1)$, and denote
	\begin{equation*}
		\mathsf{f}_{\uptau}(\upmu,z):=\sum_{k\in\mathbb{Z}}\frac{\uptau^k}{1-\uptau^k\upmu}\,z^k.
	\end{equation*}
	This series converges for $1<|z|<\uptau^{-1}$, and extends analytically to all
	$z\ne0$ with poles at $\uptau^{\Z}$.
	Also set
	$\upphi_{t}(\upzeta):=e^{t \upzeta/(1- \upzeta)}$.
	Fix any $\mathsf{r}\in(\uptau,1)$
	and define
	the kernel 
	\begin{equation}\label{kernel_J}
		\mathsf{J}^{(\upmu)}_{m,t}(\upeta,\upeta'):=\frac{1}{2\pi\mathbf{i}}\int_{|\upzeta|=\mathsf{r}'}
		\frac{\upphi_{t}(\upzeta)\upzeta^{m}}{\upphi_{t}(\upeta')(\upeta')^{m+1}}
		\frac{\mathsf{f}_{\uptau}(\upmu,\upzeta/\upeta')}{\upzeta-\upeta}\,d \upzeta,
	\end{equation}
	where $\mathsf{r}'$ is any number in 
	$(1,\mathsf{r}/\uptau)$.\footnote{%
		Here and below by a kernel we mean a function 
		of two variables belonging to a certain space. A kernel can be associated
		with an integral operator acting on functions on this space.%
	}
	For the kernel $\mathsf{J}^{(\upmu)}_{m,t}(\upeta,\upeta')$
	the variables $\upeta,\upeta'$ belong to a circle in the complex plane with
	center zero and radius $\mathsf{r}$.
	
	We will need the Fredholm determinant of the kernel \eqref{kernel_J}
	of the form $\det\bigl( \mathbf{1}+\upmu\, \mathsf{J}^{(\upmu)}_{m,t} \bigr)$.
	Here $\mathbf{1}$ is the identity operator, and this Fredholm determinant can be defined by a convergent 
	series 
	\begin{equation}\label{Fredholm_determinant}
		\det\bigl( \mathbf{1}+\upmu\, \mathsf{J}^{(\upmu)}_{m,t} \bigr)
		=
		1+
		\sum_{n=1}^{\infty}\frac{\upmu^n}{n!}\frac{1}{(2\pi\mathbf{i})^n}
		\idotsint\limits_{|\upeta_j|=\mathsf{r}}
		\mathop{\mathrm{det}}_{i,j=1}^{n}[\mathsf{J}^{(\upmu)}_{m,t}(\upeta_i,\upeta_j)]\,
		d\upeta_1 \ldots d\upeta_n.
	\end{equation}
	That is, one forms determinants of growing order out of the kernel
	$\upmu\, \mathsf{J}^{(\upmu)}_{m,t}(\upeta_i,\upeta_j)$ and integrates them over the direct powers of the 
	circle $|\upeta|=\mathsf{r}$. 
	We refer to \cite{Bornemann} for a review of Fredholm determinants.

	We will utilize \cite[Lemma 4]{TW_asymptotics} which states that 
	for the ASEP with the step initial configuration,
	\begin{equation}\label{ASEP_TW_formula}
		\mathbb{P}_{\mathrm{step},q}\left( \mathsf{x}_m(t/\upgamma)>0 \right)
		=\frac{1}{2\pi\mathbf{i}}\int_{|\upmu|=\mathsf{R}}
		(\upmu;\uptau)_{\infty}\cdot
		\det\bigl( \mathbf{1}+\upmu\, \mathsf{J}^{(\upmu)}_{m,t} \bigr)
		\frac{d\upmu}{\upmu},
	\end{equation}
	where $\mathsf{R}\in(\uptau,+\infty)\setminus
	\left\{ 1,\uptau^{-1},\uptau^{-2},\ldots  \right\}$
	is fixed and $(\upmu;\uptau)_{\infty}:=(1-\upmu)(1-\upmu\uptau)(1-\upmu\uptau^{2})\ldots $
	is the infinite $\uptau$-Pochhammer symbol.
	Note that in \eqref{ASEP_TW_formula} the sign in the left-hand side is ``$>$''
	as opposed to \cite{TW_asymptotics} because we consider the step initial configuration
	with particles packed to the left of the origin, and in \cite{TW_asymptotics} the 
	particles are packed to the right of it.
	
	\bigskip
	\noindent
	{\bf Step 2. Critical points.}
	Our goal is to understand the asymptotic behavior of 
	the right-hand side of \eqref{ASEP_TW_formula}
	as $t\to\infty$ and $m=\lfloor \frac{t}{4}(1-\varepsilon) \rfloor$ with fixed $\varepsilon\in(0,1)$
	(see Remark~\ref{rmk:ASEP_more_discussion} below for a discussion of what should be expected
	when $\varepsilon\le 0$).
	Let us focus on the integrand in \eqref{kernel_J} and rewrite it as
	\begin{equation}\label{exponents_S_action}
		\frac{\upphi_{t}(\upzeta)(-\upzeta)^{m}}{\upphi_{t}(\upeta')(-\upeta')^{m}\upeta'}
		\frac{\mathsf{f}_{\uptau}(\upmu,\upzeta/\upeta')}{\upzeta-\upeta}
		=
		\exp
		\Bigl\{ t \bigl( S(\upzeta) - S(\upeta')\bigr)\Bigr\}
		\left( \frac{\upzeta}{\upeta'} \right)^{m-\frac{t}{4}(1-\varepsilon)}
		\frac{\mathsf{f}_{\uptau}(\upmu,\upzeta/\upeta')}{\upeta'(\upzeta-\upeta)},
	\end{equation}
	where
	\begin{equation}\label{action_S}
		S(\upzeta):=\frac{\upzeta}{1-\upzeta}+
		\frac{1-\varepsilon}{4}\log(-\upzeta).
	\end{equation}
	We have changed the signs to $(-\upzeta)$ and $(-\upeta')$ for later convenience;
	note that the term $(\upzeta/\upeta')^{m-t(1-\varepsilon)/4}$ always stays bounded as $t\to\infty$
	because $m=\lfloor \frac{t}{4}(1-\varepsilon) \rfloor $.

	Having all the essential dependence on $t$ in the exponent, we employ a standard idea that
	the asymptotic behavior of the integral in \eqref{kernel_J} and subsequently
	of the whole Fredholm determinant \eqref{Fredholm_determinant} can be derived
	by the steepest descent method. The critical points
	of $S(\upzeta)$ are found from the following equation equivalent to $S'(\upzeta)=0$:
	\begin{equation*}
		(1+\upzeta)^2-\varepsilon(1-\upzeta)^2=0,
		\qquad\qquad 
		\upzeta^{(1)}=\frac{\sqrt\varepsilon-1}{\sqrt\varepsilon+1},
		\quad
		\upzeta^{(2)}=\frac{1}{\upzeta^{(1)}}=\frac{\sqrt\varepsilon+1}{\sqrt\varepsilon-1}.
	\end{equation*}
	When $\varepsilon=0$ these two roots coincide (and are both equal to $-1$)
	leading to the GUE Tracy--Widom asymptotics
	derived in \cite{TW_asymptotics}. 
	When $\varepsilon>0$ the roots $\upzeta^{(1)}, \upzeta^{(2)}$ are real and distinct. 
	Moreover, they satisfy
	$\upzeta^{(1)}\in(-1,0)$, $\upzeta^{(2)}\in(-\infty,-1)$.
	For future convenience we need the following expressions:
	\begin{equation*}
		S''(\upzeta^{(1)})=
		-
		\frac{(1+\sqrt{\varepsilon})^3 \sqrt{\varepsilon}}{4 (1-\sqrt\varepsilon)}<0
		,
		\qquad 
		S''(\upzeta^{(2)})=
		\frac{(1-\sqrt{\varepsilon})^3 \sqrt{\varepsilon}}{4 (1+\sqrt\varepsilon)}>0.
	\end{equation*}
	Let us define
	\begin{equation}
		\label{tilde_I_epsilon}
		\hat \Phi_+(\varepsilon):=
		S(\upzeta^{(1)})-S(\upzeta^{(2)})
		=
		\sqrt{\varepsilon}-(1-\varepsilon)\tanh^{-1}(\sqrt\varepsilon),
	\end{equation}
	where the expression in the right-hand side is a straightforward computation.
	Since $\hat \Phi_+(0)=0$ and $\frac{\partial}{\partial\varepsilon}\hat \Phi_+(\varepsilon)=\tanh^{-1}(\sqrt\varepsilon)>0$
	which is positive for $0<\varepsilon<1$ and behaves as $\sim\sqrt\varepsilon$ as $\varepsilon\to0+$,
	we have
	\begin{equation*}
		\text{$\hat \Phi_+(\varepsilon)>0$ for $0<\varepsilon<1$,
		\qquad 
		$\hat \Phi_+(\varepsilon)\sim\frac{2}{3}\varepsilon^{\frac{3}{2}}$ as $\varepsilon\to0+$.}
	\end{equation*}

	Let us now choose the radii $\mathsf{r}$ and $\mathsf{r}'$ in 
	\eqref{kernel_J}, \eqref{Fredholm_determinant}
	so that the contours for $\upzeta$ and the $\upeta_j$'s pass
	through our single critical points $\upzeta^{(1)},\upzeta^{(2)}$. 
	Namely, set
	\begin{equation}
		\label{radii_choice}
		\mathsf{r}:=|\upzeta^{(1)}|,\qquad 
		\mathsf{r}':=|\upzeta^{(2)}|.
	\end{equation}
	For the conditions $\mathsf{r}\in(\uptau,1)$ and $\mathsf{r}'\in(1,\mathsf{r}/\uptau)$
	to hold for \eqref{radii_choice} we need to assume that $\varepsilon$ is not too large,
	namely,
	\begin{equation}\label{bound_on_epsilon}
		0<\varepsilon<\varepsilon^\circ,\qquad \qquad 
		\varepsilon^\circ:=\left( \frac{1-\sqrt\uptau}{1+\sqrt\uptau} \right)^2.
	\end{equation}
	One can readily check that \eqref{bound_on_epsilon} implies 
	$\uptau<\mathsf{r}<1$ and $1<\mathsf{r}'<\mathsf{r}/\uptau$.
	
	Because the particles in the ASEP are ordered,
	the claim of Theorem~\ref{thm:ASEP_theorem_intro} 
	for $\varepsilon\in(0,\varepsilon^\circ)$
	would also imply the claim for 
	$\varepsilon\in(\varepsilon^\circ,1)$ 
	if we truncate the function $\hat \Phi_+(\varepsilon)$ \eqref{tilde_I_epsilon}.
	Namely, let us define
	$\Phi_+(\varepsilon)$ to be equal to 
	$\hat \Phi_+(\varepsilon)$ for 
	$0<\varepsilon<\varepsilon^\circ$ and to 
	$\hat \Phi_+(\varepsilon^\circ)$ otherwise, that is,
	\begin{equation}
		\label{I_epsilon}
		\Phi_+(\varepsilon):=
		\begin{cases}
			\sqrt{\varepsilon}-(1-\varepsilon)\tanh^{-1}(\sqrt\varepsilon),& 0<\varepsilon<
			\varepsilon^\circ
			;\\
			\sqrt{\varepsilon^\circ}-(1-\varepsilon^\circ)\tanh^{-1}(\sqrt{\varepsilon^\circ})
			,&
			\varepsilon^\circ
			\le \varepsilon<1.
		\end{cases}
	\end{equation}
	Thus defined $\Phi_+(\varepsilon)$ is weakly increasing in $\varepsilon\in(0,1)$.
	Therefore, throughout the rest of the proof we can and will assume 
	that $\varepsilon$ is bounded as in 
	\eqref{bound_on_epsilon}.
	
	\bigskip
	\noindent
	{\bf Step 3. Estimates of the real part of $S$.}
	Our next goal is to estimate the real part of the function $S$ \eqref{action_S}
	on the circles
	with the radii \eqref{radii_choice}. 
	Namely, let us show that for any $\varepsilon\in(0,1)$:
	\begin{equation}\label{max_min_steepest_estimates}
		\parbox{.9\textwidth}{\begin{enumerate}
			\item For any $\upzeta\in \mathbb{C}$ with $|\upzeta|=\mathsf{r}=|\upzeta^{(1)}|$, $\upzeta\ne\upzeta^{(1)}$
				we have $\Re S(\upzeta)>\Re S(\upzeta^{(1)})$.
			\item For any $\upzeta\in \mathbb{C}$ with $|\upzeta|=\mathsf{r}'=|\upzeta^{(2)}|$, $\upzeta\ne\upzeta^{(2)}$
				we have $\Re S(\upzeta)<\Re S(\upzeta^{(2)})$.
		\end{enumerate}}
	\end{equation}
	These estimates follow from the 
	straightforward computations:
	\begin{equation*}
		\frac{\partial}{\partial\uptheta}\Re S(\upzeta^{(1)}e^{\mathbf{i}\uptheta})=
		\frac{(1-\varepsilon) \sqrt{\varepsilon} \sin \uptheta }
		{(1+\varepsilon+(1-\varepsilon)\cos\uptheta)^2}>0,\qquad \uptheta\in(0,\pi),
	\end{equation*}
	and 
	\begin{equation*}
		\frac{\partial}{\partial\uptheta}\Re S(\upzeta^{(2)}e^{\mathbf{i}\uptheta})=
		-
		\frac{(1-\varepsilon) \sqrt{\varepsilon} \sin \uptheta }
		{(1+\varepsilon+(1-\varepsilon)\cos\uptheta)^2}<0,\qquad \uptheta\in(0,\pi)
		.
	\end{equation*}
	These signs of these derivatives establish \eqref{max_min_steepest_estimates}
	for $\upzeta$ in the lower half plane. 
	To obtain these inequalities in the upper half plane one needs to take 
	$\upzeta=-\upzeta^{(1)}e^{\mathbf{i}\uptheta}$ (and similarly for $\upzeta^{(2)}$)
	which leads to the opposite signs of the derivatives.
	See Figure~\ref{fig:epsilon_steepest_descent} for an illustration.
	\begin{figure}[htpb]
		\centering
		\includegraphics[width=.49\textwidth]{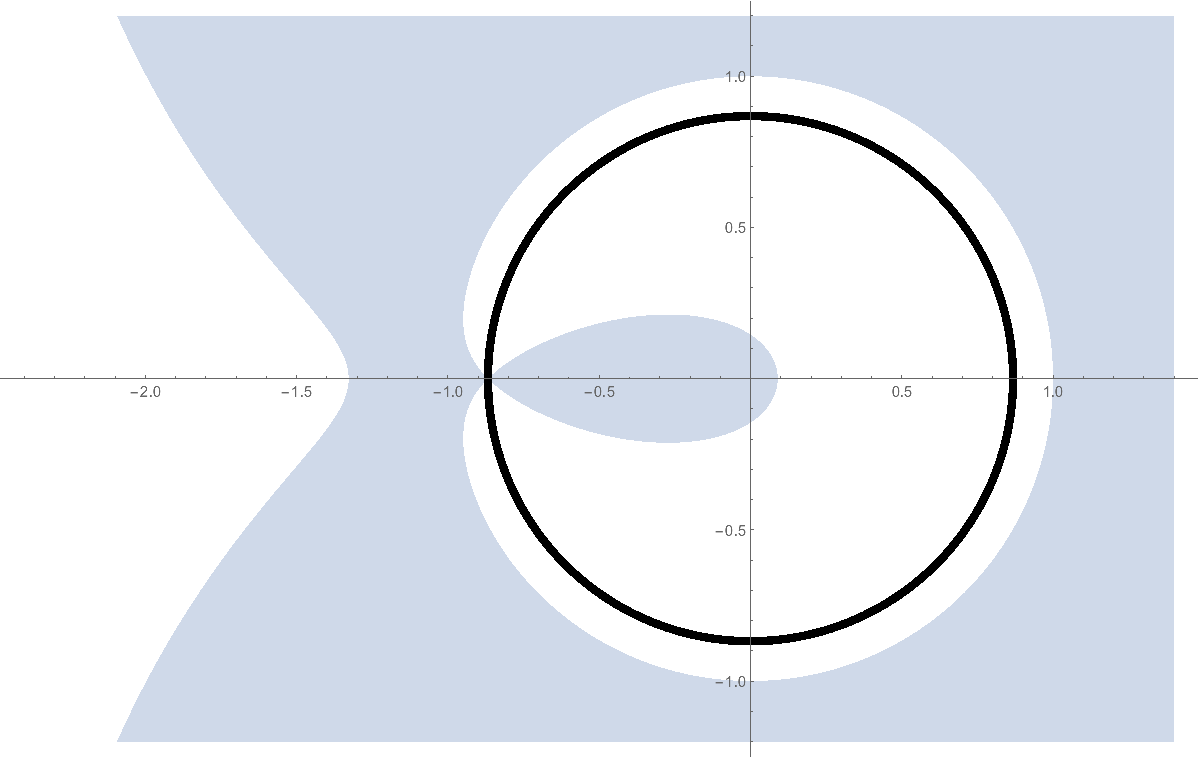}
		\includegraphics[width=.49\textwidth]{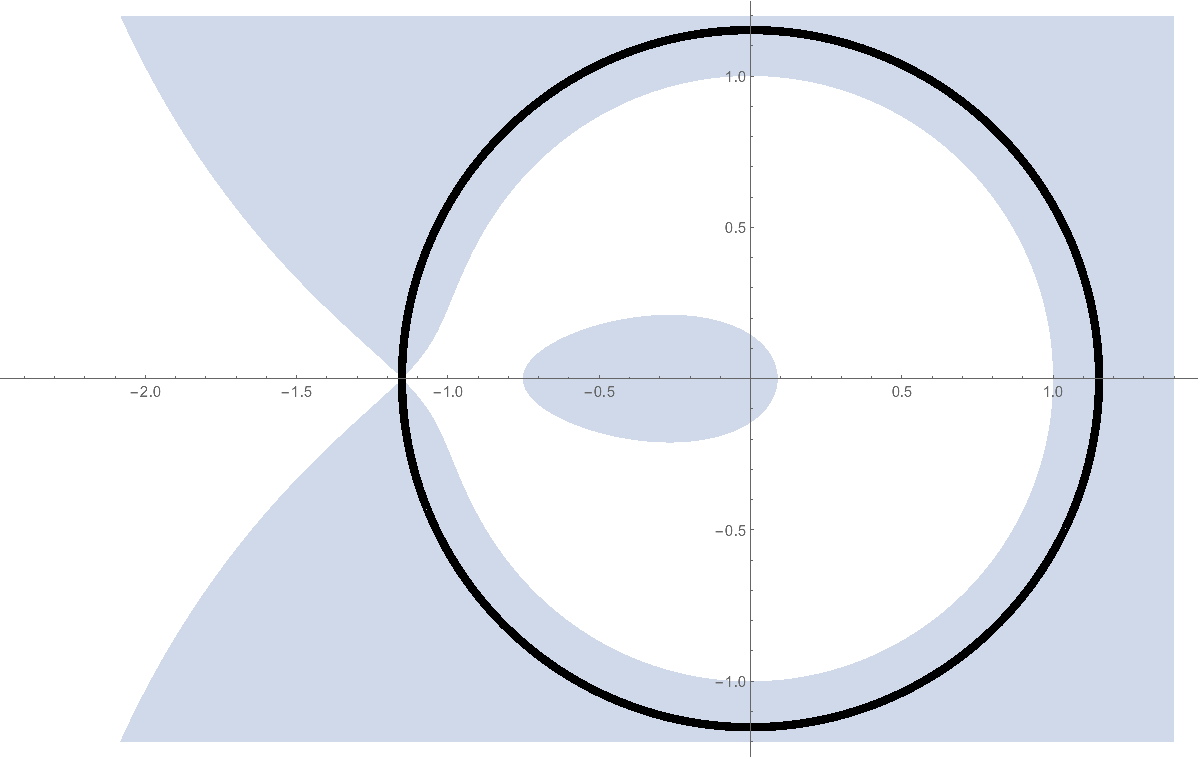}
		\caption{%
			The circles $|\upzeta|=|\upzeta^{(i)}|$, where $i=1$ on the left and 
			$i=2$ on the right.
			Shaded are the regions where $\Re S(\upzeta)<\Re S(\upzeta^{(i)})$.
			The circles intersect the negative real line at points $\upzeta^{(i)}$.%
		}
		\label{fig:epsilon_steepest_descent}
	\end{figure}

	\bigskip
	\noindent
	{\bf Step 4. Steepest descent asymptotics of $\mathsf{J}^{(\upmu)}_{m,t}(\upeta,\upeta')$.}
	Fix $\upeta,\upeta'$ on a circle of radius $\mathsf{r}$.
	In this step we argue that the main contribution to the asymptotics of 
	$\mathsf{J}^{(\upmu)}_{m,t}(\upeta,\upeta')$
	as an integral 
	over $\upzeta$ with $|\upzeta|=\mathsf{r}'$
	comes from a small neighborhood of $\upzeta^{(2)}$.
	All error terms we obtain here and in the next step
	are uniform in $\upmu$ because the contour for $\upmu$
	is chosen so that $|\mathsf{f}_{\uptau}(\upmu,\upzeta/\upeta')|$
	is bounded uniformly in $\upmu$.
	
	The $\upzeta$-dependence in the integrand in
	$\mathsf{J}^{(\upmu)}_{m,t}(\upeta,\upeta')$
	has the form
	\begin{equation}
		\label{eq:zeta_part_of_the_integrand_ASEP}
		e^{t S(\upzeta)}
		\upzeta^{m-\frac{t}{4}(1-\varepsilon)}
		\frac{\mathsf{f}_{\uptau}(\upmu,\upzeta/\upeta')}{\upzeta-\upeta},
	\end{equation}
	see 
	\eqref{exponents_S_action}.
	Take a neighborhood of $\upzeta^{(2)}$ of size $t^{-3/8}$.
	Let us use
	\eqref{max_min_steepest_estimates} and the estimates of the derivatives 
	in the previous step
	to bound the absolute value of the 
	integral of 
	\eqref{eq:zeta_part_of_the_integrand_ASEP}
	over the 
	part of the circle $|\upzeta|=\mathsf{r}'$
	outside this neighborhood of $\upzeta^{(2)}$.
	We have 
	\begin{align*}
		\bigl|\eqref{eq:zeta_part_of_the_integrand_ASEP}\bigr|
		&
		\le
		\exp
		\bigl[ t\max
			\bigl\{ 
				\Re S(\upzeta)- S(\upzeta^{(2)}) \colon |\upzeta|=\mathsf{r}',\;
				|\upzeta-\upzeta^{(2)}|>t^{-3/8} 
			\bigr\} 
		\bigr]
		e^{t S(\upzeta^{(2)})}
		|\upzeta|^{m-\frac{t}{4}(1-\varepsilon)}
		\left|
			\frac{\mathsf{f}_{\uptau}(\upmu,\upzeta/\upeta')}{\upzeta-\upeta}
		\right|
		\\&
		\le
		\mathsf{C} 
		e^{-\mathsf{c}t^{5/8}}
		\cdot
		e^{t S(\upzeta^{(2)})}
	\end{align*}
	for some $\mathsf{C}, \mathsf{c}>0$ (recall that $m-\frac{t}{4}(1-\varepsilon)$ is at most one in absolute value, so 
	we simply estimate this power of $|\upzeta|$ by a constant).
	Inside the $t^{-3/8}$-neighborhood of $\upzeta^{(2)}$ make a change of variables
	\begin{equation*}
		\upzeta=\upzeta^{(2)}-\mathbf{i}\frac{u}{\sqrt t},\qquad  |u|<t^{1/8}.
	\end{equation*}
	Here
	the minus sign in the second term accounts for the direction of the contour in the neighborhood
	of $\upzeta^{(2)}$, and the bound $|u|<t^{1/8}$ corresponds to the size 
	$t^{-3/8}$ of the neighborhood.
	We have
	\begin{equation*}
		tS(\upzeta)
		=
		tS(\upzeta^{(2)})
		+
		t\,\frac{S''(\upzeta^{(2)})}{2}( \upzeta-\upzeta^{(2)} )^2
		+
		O\bigl(t(\upzeta-\upzeta^{(2)})^3\bigr)
		=
		tS(\upzeta^{(2)})
		-
		\frac{S''(\upzeta^{(2)})}{2}\,u^2+O(t^{-1/8}).
	\end{equation*}
	Using \eqref{exponents_S_action} 
	and the fact that $S''(\upzeta^{(2)})>0$, 
	we can rewrite the integral in
	$\mathsf{J}^{(\upmu)}_{m,t}$ \eqref{kernel_J} 
	as 
	\begin{align*}
		\mathsf{J}^{(\upmu)}_{m,t}(\upeta,\upeta')
		&=
		\frac{1+O(e^{-\mathsf{c}t^{5/8}})}{2\pi\mathbf{i}}
		\,e^{tS(\upzeta^{(2)})-tS(\upeta')}
		\int\limits_{\substack{|\upzeta|=\mathsf{r}'\\|\upzeta-\upzeta^{(2)}|<t^{-3/8}}}
		e^{ t S(\upzeta) - tS(\upzeta^{(2)})}
		\left( \frac{\upzeta}{\upeta'} \right)^{m-\frac{t}{4}(1-\varepsilon)}
		\frac{\mathsf{f}_{\uptau}(\upmu,\upzeta/\upeta')}{\upeta'(\upzeta-\upeta)}
		\,d \upzeta
		\\&=
		\frac{1+O(t^{-1/8})}{2\pi\mathbf{i}}
		\,e^{tS(\upzeta^{(2)})-tS(\upeta')}
		\left( \frac{\upzeta^{(2)}}{\upeta'} \right)^{m-\frac{t}{4}(1-\varepsilon)}
		\frac{\mathsf{f}_{\uptau}(\upmu,\upzeta^{(2)}/\upeta')}{\upeta'(\upzeta^{(2)}-\upeta)}
		\int_{-\infty}^{\infty} \frac{(-\mathbf{i})e^{-\frac12S''(\upzeta^{(2)})u^2}}{\sqrt t} 
		\,du.
	\end{align*}
	The last integral is a convergent Gaussian integral, and thus we obtain
	\begin{equation}\label{J_asymptotics}
		\mathsf{J}^{(\upmu)}_{m,t}(\upeta,\upeta')
		=
		-\frac{1+O(t^{-1/8})}{\sqrt{2\pi t S''(\upzeta^{(2)})}}
		\,e^{tS(\upzeta^{(2)})-tS(\upeta')}
		\left( \frac{\upzeta^{(2)}}{\upeta'} \right)^{m-\frac{t}{4}(1-\varepsilon)}
		\frac{\mathsf{f}_{\uptau}(\upmu,\upzeta^{(2)}/\upeta')}{\upeta'(\upzeta^{(2)}-\upeta)}.
	\end{equation}
	Note that the constant in $O(t^{-1/8})$ can be taken independent of 
	$\upmu,\upeta,\upeta'$ on our contours
	because the quantity 
	$\bigl|(\upzeta/\upeta')^{m-\frac{t}{4}(1-\varepsilon)}
	\frac{\mathsf{f}_{\uptau}(\upmu,\upzeta/\upeta')}{\upeta'(\upzeta-\upeta)}\bigr|$
	is bounded away from zero and infinity.

	\smallskip
	\noindent
	{\bf Step 5. Asymptotics of the Fredholm determinant.}
	We see that \eqref{J_asymptotics} approximates
	$\mathsf{J}^{(\upmu)}_{m,t}(\upeta,\upeta')$ 
	(viewed as an operator) as $t\to\infty$
	by a rank one operator.
	Therefore, the $n\times n$ determinants entering the Fredholm determinant 
	\eqref{Fredholm_determinant} are simplified as 
	\begin{equation*}
		\mathop{\mathrm{det}}_{i,j=1}^{n}[\mathsf{J}^{(\upmu)}_{m,t}(\upeta_i,\upeta_j)]
		=
		\mathop{\mathrm{det}}_{i,j=1}^{n}[1+t^{-\frac18}\tilde{\mathsf{J}}^{(\upmu)}_{m,t}(\upeta_i,\upeta_j)]
		\cdot
		\prod_{i=1}^{n}
		\frac{-e^{tS(\upzeta^{(2)})-tS(\upeta_i)}}{\sqrt{2\pi t S''(\upzeta^{(2)})}}
		\left( \frac{\upzeta^{(2)}}{\upeta_i} \right)^{m-\frac{t}{4}(1-\varepsilon)}
		\frac{\mathsf{f}_{\uptau}(\upmu,\upzeta^{(2)}/\upeta_i)}{\upeta_i(\upzeta^{(2)}-\upeta_i)},
	\end{equation*}
	where the terms $t^{-\frac18}\tilde{\mathsf{J}}^{(\upmu)}_{m,t}(\upeta_i,\upeta_j)$
	correspond to $O(t^{-1/8})$ in \eqref{J_asymptotics}.\footnote{%
		Note that here ``$1$'' in the determinant in the right-hand side is
		simply the constant $1$ entering every matrix element, and not the identity operator.%
	}
	Applying Lemma~\ref{lemma:rank_one_approximation} (see below) to the determinant
	in the right-hand side we see that it is bounded by $B^nn^{n/2+1}t^{-\frac{1}{8}(n-1)}$,
	where $B$ is a bound on $\tilde{\mathsf{J}}^{(\upmu)}_{m,t}$.
	Thus, the whole Fredholm determinant \eqref{Fredholm_determinant} can be rewritten as
	\begin{equation}
		\label{Fredholm_estimate_1}
		\det\bigl( \mathbf{1}+\upmu\, \mathsf{J}^{(\upmu)}_{m,t} \bigr)
		=
		1+
		\frac{\upmu}{2\pi\mathbf{i}}\int_{|\upeta|=\mathsf{r}}
		\mathsf{J}^{(\upmu)}_{m,t}(\upeta,\upeta)\,d\upeta
		+\mathsf{Remainder},
	\end{equation}
	where
	\begin{equation}
		\label{Fredholm_estimate_2}
		\bigl|\mathsf{Remainder}\bigr|\le
		\sum_{n=2}^{\infty}
		\frac{(B\upmu)^nn^{n/2+1}}{t^{\frac{1}{8}(n-1)}n!}\cdot
		\left(
			\frac{1}{2\pi\mathbf{i}}\int_{|\upeta|=\mathsf{r}}
			\left|
			\frac{e^{tS(\upzeta^{(2)})-tS(\upeta)}}{\sqrt{2\pi t S''(\upzeta^{(2)})}}
			\left( \frac{\upzeta^{(2)}}{\upeta} \right)^{m-\frac{t}{4}(1-\varepsilon)}
			\frac{\mathsf{f}_{\uptau}(\upmu,\upzeta^{(2)}/\upeta)}{\upeta(\upzeta^{(2)}-\upeta)}\right|d\upeta
		\right)^n.
	\end{equation}
	Both the main term corresponding to $n=1$ in \eqref{Fredholm_estimate_1}
	and the integrals in the remainder in \eqref{Fredholm_estimate_2}
	can be estimated using the steepest descent method with the 
	help of \eqref{max_min_steepest_estimates} similarly to Step 4 above.

	First, for the single integral in \eqref{Fredholm_estimate_1} we have
	(in particular, using the fact that the constant in $O(t^{-1/8})$ is independent
	of $\upeta$)
	\begin{equation*}
		\frac{\upmu}{2\pi\mathbf{i}}
		\int_{|\upeta|=\mathsf{r}}\mathsf{J}^{(\upmu)}_{m,t}(\upeta,\upeta)\,d\upeta
		=
		-\frac{\upmu(1+O(t^{-1/8}))}{2\pi\mathbf{i}\sqrt{2\pi t S''(\upzeta^{(2)})}}
		\int_{|\upeta|=\mathsf{r}}
		e^{tS(\upzeta^{(2)})-tS(\upeta)}
		\left( \frac{\upzeta^{(2)}}{\upeta} \right)^{m-\frac{t}{4}(1-\varepsilon)}
		\frac{\mathsf{f}_{\uptau}(\upmu,\upzeta^{(2)}/\upeta)}{\upeta(\upzeta^{(2)}-\upeta)}
		\,d\upeta.
	\end{equation*}
	The main contribution to the integral 
	over $|\upeta|=\mathsf{r}$ 
	comes from a small neighborhood of the critical point 
	$\upzeta^{(1)}$ (recall that $\mathsf{r}=|\upzeta^{(1)}|$).
	Making a change of variables $\upeta=\upzeta^{(1)}-\mathbf{i}\frac{u}{\sqrt t}$
	leads to a convergent Gaussian integral of $e^{\frac{1}{2}S''(\upzeta^{(1)})u^2}$;
	note that $S''(\upzeta^{(1)})<0$. Therefore, we can continue
	as
	\begin{align*}
		\frac{\upmu}{2\pi\mathbf{i}}
		\int_{|\upeta|=\mathsf{r}}\mathsf{J}^{(\upmu)}_{m,t}(\upeta,\upeta)\,d\upeta
		&=
		\frac{\upmu e^{-t\hat \Phi_+(\varepsilon)}(1+O(t^{-1/8}))}
		{2\pi t\sqrt{ |S''(\upzeta^{(1)})|S''(\upzeta^{(2)})}}
		\left( \frac{\upzeta^{(2)}}{\upzeta^{(1)}} \right)^{m-\frac{t}{4}(1-\varepsilon)}
		\frac{\mathsf{f}_{\uptau}(\upmu,\upzeta^{(2)}/\upzeta^{(1)})}{(-\upzeta^{(1)})(\upzeta^{(1)}-\upzeta^{(2)})}
		\\&=
		\frac{\upmu e^{-t\hat \Phi_+(\varepsilon)}(1+O(t^{-1/8}))}
		{2\pi t \varepsilon}
		\left( \frac{1+\sqrt\varepsilon}{1-\sqrt\varepsilon} \right)^{1+2m-\frac{t}{2}(1-\varepsilon)}
		\mathsf{f}_{\uptau}(\upmu,\upzeta^{(2)}/\upzeta^{(1)}).
	\end{align*}
	Similarly to Step 4, the constant in $O(t^{-1/8})$ can be taken independent of $\upmu$.
	
	In the remainder \eqref{Fredholm_estimate_2} we can similarly bound each integral 
	by $B_1t^{-1}e^{-t\hat \Phi_+(\varepsilon)}$ (with $B_1$ independent of $\upmu$).
	Therefore, the series in \eqref{Fredholm_estimate_2} converges thanks to the factorial in the denominator, and its sum
	behaves as $O(t^{-\frac{17}{8}}e^{-2t\hat \Phi_+(\varepsilon)})$,
	which is exponentially negligible compared to the main contribution
	in \eqref{Fredholm_estimate_1}.

	\bigskip
	\noindent
	{\bf Step 6. Completing the proof.}
	Putting all together and using \eqref{ASEP_TW_formula} we see that
	\begin{align*}
		&\mathbb{P}_{\mathrm{step},q}\left( \mathsf{x}_m(t/\upgamma)>0 \right)
		\\&=
		\frac{1}{2\pi\mathbf{i}}\int_{|\upmu|=\mathsf{R}}
		(\upmu;\uptau)_{\infty}
		\left( 
			1
			+
			\frac{\upmu e^{-t\hat \Phi_+(\varepsilon)}}
			{2\pi t \varepsilon}
			\left( \frac{1+\sqrt\varepsilon}{1-\sqrt\varepsilon} \right)^{1+2m-\frac{t}{2}(1-\varepsilon)}
			\mathsf{f}_{\uptau}(\upmu,\upzeta^{(2)}/\upzeta^{(1)}) \bigl(1+O(t^{-1/8})\bigr)
		\right)
		\frac{d\upmu}{\upmu},
	\end{align*}
	where the remainder \eqref{Fredholm_estimate_2} (the sum of the terms with $n\ge2$) 
	is also incorporated into the $O(t^{-1/8})$ term which can be taken independent of $\upmu$. 
	Recall that $m=\lfloor \frac{t}{4}(1-\varepsilon) \rfloor $, so the term 
	$\left( \frac{1+\sqrt\varepsilon}{1-\sqrt\varepsilon} \right)^{1+2m-\frac{t}{2}(1-\varepsilon)}$
	stays bounded as $t\to\infty$.
	Observe that 
	\begin{equation*}
		\frac{1}{2\pi\mathbf{i}}\int_{|\upmu|=\mathsf{R}}
		(\upmu;\uptau)_{\infty}
		\frac{d\upmu}{\upmu}=1
	\end{equation*}
	as the residue at zero, 
	and that 
	\begin{equation*}
		\frac{1}{2\pi\mathbf{i}}\int_{|\upmu|=\mathsf{R}}
		(\upmu;\uptau)_{\infty}\,
		\mathsf{f}_{\uptau}(\upmu,\upzeta^{(2)}/\upzeta^{(1)})\,d\upmu
		=-\sum_{k=1}^{\infty}(\uptau^k;\uptau)_{\infty}(\upzeta^{(1)}/\upzeta^{(2)})^k=
		-
		\frac{\upzeta^{(1)}/\upzeta^{(2)}(\uptau;\uptau)_\infty}{(\upzeta^{(1)}/\upzeta^{(2)};\uptau)_{\infty}}
	\end{equation*}
	by residues using the definition of $\mathsf{f}_{\uptau}$,
	and by the $q$-binomial theorem. 
	Note that the latter quantity is real and negative as should be.
	This completes the proof of the large deviation bound of 
	Theorem~\ref{thm:ASEP_theorem_intro}.
\end{proof}

The following lemma is employed in the above proof of Theorem~\ref{thm:ASEP_theorem_intro}:
\begin{lemma}
	\label{lemma:rank_one_approximation}
	Let $J(\upeta,\upeta')
	=
	1
	+
	t^{-\updelta}\tilde J(\upeta,\upeta')$,
	where $\updelta>0$.
	Then 
	for any $n\ge1$ and all $t>0$ large enough we have
	\begin{equation*}
		\left|
			\mathop{\mathrm{det}}_{i,j=1}^{n}\left[ J(\upeta_i,\upeta_j) \right]
		\right|
		\le 
		\frac{B^n n^{n/2+1}}{t^{\updelta(n-1)}},
	\end{equation*}
	where $B=1+\max\bigl\{ |\tilde J(\upeta_i,\upeta_j)| \colon 1\le i,j\le n \bigr\}$.
\end{lemma}
\begin{proof}
	First, due to the rank one part in $J$ (i.e., to the matrix
	consisting of all $1$'s), the terms of orders $1,t^{-\updelta},\ldots,t^{-(n-2)\updelta}$
	in $t$ cancel out. 
	This leaves only two powers of $t$, $t^{- \updelta(n-1)}$ and $t^{- \updelta n}$, so
	\begin{equation*}
		\mathop{\mathrm{det}}_{i,j=1}^{n}\left[ J(\upeta_i,\upeta_j) \right]
		=
		t^{- \updelta(n-1)}
		\mathsf{D}_n
		+
		t^{- \updelta n}
		\mathop{\mathrm{det}}_{i,j=1}^{n} \bigl[ \tilde J(\upeta_i,\upeta_j) \bigr].
	\end{equation*}
	Here $\mathsf{D}_n$ is a sum of $n$ determinants of 
	$\tilde J(\upeta_i,\upeta_j)$ with one of the rows replaced by the row of ones.
	Estimating the absolute value of each of these determinants 
	by Hadamard's inequality 
	and noting that for large $t$ the first summand above dominates,
	we get the desired bound.
\end{proof}


Let us conclude with two comments on Theorem~\ref{thm:ASEP_theorem_intro} and its proof given above:

\begin{remark}
	\label{rmk:ASEP_more_discussion}
	1. A similar approach 
	can be utilized to obtain a one-sided
	large deviation bound of the form
	\begin{equation*}
		\mathbb{P}_{\mathrm{step},q}
		\bigl( 
			\mathsf{x}_{\lfloor \mathsf{s} t (1-\varepsilon) \rfloor}(t/\upgamma)
			< (-1+2\sqrt{\mathsf{s}})t 
		\bigr)
		\le 
		\mathsf{C} e^{-t \Phi_+^{\mathsf{s}}(\varepsilon)},
	\end{equation*}
	where $\mathsf{s}\in(0,1)$ is fixed. 
	This is because \cite[Lemma 4]{TW_asymptotics}
	provides a pre-limit Fredholm determinantal formula for the probability 
	$\mathbb{P}_{\mathrm{step},q}\left( \mathsf{x}_m(t/\upgamma)<x \right)$
	for any $m,x,t$ (with a suitably modified kernel $\mathsf{J}^{(\upmu)}_{m,t}$).
	The function $\Phi_+^{\mathsf{s}}(\varepsilon)$ may be explicitly computed using this formula.
	However, for the analysis of the coarsening model we only need the 
	particular case $\mathsf{s}=\frac{1}{4}$ which is our
	Theorem~\ref{thm:ASEP_theorem_intro}.

\smallskip

	2. Our proof of Theorem~\ref{thm:ASEP_theorem_intro} shows that, up to polynomial corrections,
	the probability 
	\begin{equation*}
		\mathbb{P}_{\mathrm{step},q}\left( \mathsf{x}_{\lfloor t(1-\varepsilon)/4 \rfloor }(t/\upgamma)<0 \right)
	\end{equation*}
	that the ASEP is ``too slow'' goes to zero at exponential rate. 
	The large deviation probability 
	\begin{equation*}
		\mathbb{P}_{\mathrm{step},q}\left( \mathsf{x}_{\lfloor t(1+\varepsilon)/4 \rfloor }(t/\upgamma)>0 \right)
	\end{equation*}
	at the other tail (that the ASEP is ``too fast'') 
	should go to zero at a much faster rate $\exp\{-t^2 \Phi_-(\varepsilon)\}$ 
	for some other rate function $\Phi_-(\varepsilon)>0$
	(e.g., see \cite{Joh} for the case of TASEP).
	The analysis required to 
	establish this other tail bound for ASEP by the same method as Theorem~\ref{thm:ASEP_theorem_intro} 
	would likely be much more involved. 
	Indeed, in this case the Fredholm determinant would need to 
	go to zero instead of one, and because the 
	critical points have a 
	completely different 
	structure (there are two complex conjugate critical points
	of $S$ with $\Re S$ being the same at both of them),
	the result would require a much more subtle analysis of all 
	terms of the Fredholm expansion 
	\eqref{Fredholm_determinant}.
	We do not pursue this superexponential tail here.
\end{remark}

\section{Application of Fontes-Schonmann-Sidoravicius}\label{sec: FSS_sec}
In this section we prove a variant of the fixation theorem of Fontes-Schonmann-Sidoravicius (FSS). It states that if the biased coarsening model runs starting from an initial condition in which blocks of size $L_0$ are monochromatic and independent of one another, then if the probability that a block begins in the all $+1$ state is high enough, then all spins converge to $+1$ quickly.

We will begin with $L_0\ge 1$ and split the lattice into disjoint translates of the box $\Lambda_{L_0}$. Each box will be filled with $+1$ spins with probability $1-\epsilon_0$, and filled with $-1$ spins otherwise, independently of each other. Then we will run the zero-temperature Glauber dynamics on the original lattice with tie-breaking probability $q \in [0,1]$. The corresponding probability measure on the dynamics and initial distribution is denoted $\mathbb{P}_{\epsilon_0,q,L_0}$.

We will make the following assumption, and its validity will depend on the value of $q$ we use.
\begin{assumption}\label{assumption}
Let $T$ be the time for the configuration in the $L^d$ rectangle $\Lambda_L$ to reach all $+1$, when the dynamics (with $q$-biased tie breaking) is run with an initial configuration of all $-1$ inside the rectangle and all $+1$ outside.  There exist $C,\gamma \in (0,\infty)$ and $\alpha \geq 1$ such that
\begin{equation}\label{eq: assumption_erosion}
\mathbb{P}_q(T>CL^\alpha) \leq e^{-\gamma L} \text{ for all }L.
\end{equation}
\end{assumption}

Note that the measure $\mathbb{P}_q$ only depends on $q$, since the initial condition is deterministic. Furthermore, by attractiveness, if \eqref{eq: assumption_erosion} holds for some $q,C,\gamma,$ and $\alpha$ then it holds for $q' \geq q$ with the same $C,\gamma,\alpha$. Last, we note here that by Theorem~\ref{exponential box decay}, the above assumption holds in our coarsening model for $q>1/2$ in either of the following cases: $d \leq 4$ and $\alpha = d-1$, or $d\geq 5$ and any $\alpha >3$.

The main result is:
\begin{theorem}\label{thm: block_FSS}
Suppose $q \in [0,1]$ is such that \eqref{eq: assumption_erosion} holds for some $\alpha \geq 1$ and constants $C,\gamma$. 
\begin{enumerate}
\item If $\alpha=1$, and $L_0 \geq 4$ is given, there exists $\epsilon> 0$ with the following property. For any $\epsilon_0 \in (0,\epsilon)$, there exists $C_1>0$ such that
\[
\mathbb{P}_{\epsilon_0,q,L_0}(\sigma_0^s = -1 \text{ for some } s \geq t) \leq \exp\left( -C_1 \frac{t}{\log^2 t} \right) \text{ for all large } t > 0.
\]
\item If $\alpha>1$, and $\delta>0$ is given, there exists $\epsilon>0$ with the following property. For any $\epsilon_0 \in (0,\epsilon)$ and  
\[
L_0 \leq K \epsilon_0^{-\frac{2\delta}{(d-1)(\alpha-1)}},
\]
(where the constant $K$ is an explicit function of $d$ and $\alpha$), there exists $C_2>0$ such that
\[
\mathbb{P}_{\epsilon_0,q,L_0}(\sigma_0^s = -1 \text{ for some }s \geq t) \leq \exp\left( -C_2 \frac{t^{1/\alpha}}{\log^{\alpha + 3\delta} t}\right) \text{ for all large } t > 0.
\]
\end{enumerate}
\end{theorem}

\begin{remark}
One can argue by attractiveness that item 1 holds for all $L_0 \geq 1$. Part 1 of the above theorem can be improved, replacing $\log^2$ by $\log \cdot (\log \log)^2$, or more iterated logarithms. Part 2 can also be improved, replacing $\log^{\alpha + 3\delta}$ by $\log^{\alpha+2\delta} \cdot (\log \log)^{2\alpha+5\delta}$, or more iterated logarithms.
\end{remark}
\begin{remark}
Our proof is a modification of that of Fontes-Schonmann-Sidoravicius (FSS)~\cite{FSS}, allowing for a more general erosion rate.  Their original arguments used $\alpha=d$, which comes from a comparison to the Symmetric Exclusion Process (SEP).  Morris~\cite{Morris} adapted the arguments of FSS for large dimensions using the SEP estimates, and needed to choose the initial scale $(2L_0)^{d^2}> C/\eps_0$ (see Theorem 2 in~\cite{Morris}). This is roughly our dependence between $L_0$ and $\epsilon_0$ in item 2 above when $\alpha=d$ and $\delta$ is fixed. Our parameters in the $\alpha=1$ case show that one can choose $L_0$ independent of $\epsilon_0$. 

The reason why in the case $\alpha>1$, the proof requires $L_0$ to depend on $\epsilon_0$ is seen by attempting to take $\epsilon_0$ fixed but $L_0 \to \infty$ (and obtaining a contradiction) in the following bounds. From \eqref{eq: t_k_condition}, one has $t_{k+1} \geq C(n_kL_k)^\alpha \geq CL_k^\alpha$ (as $n_k \geq 1$). From \eqref{eq: time_upper_bound}, we also find $t_{k+1} \leq C'L_{k+1} = C'L_k l_{k+1}$, and so 
\[
C'' L_k^{\alpha-1} \leq l_{k+1}.
\]
For $k=0$, this becomes $l_1 \geq C''L_0^{\alpha-1}$. On the other hand, the term in \eqref{eq: end_of_bootstrapping} is a probability bound, so it is only useful if it is bounded by 1: for $k=0$,
\[
1 \geq \left( \frac{5 n_0 l_1}{3} \right)^d (2n_0^{d-1}\widetilde \epsilon_0)^{\lfloor n_0/3 \rfloor} \geq (C''' L_0)^{d(\alpha-1)} (2n_0^{d-1}\epsilon_0)^{\lfloor n_0/3 \rfloor}.
\]
The second term of this product on the right side is minimized when $n_0$ is the integer part of a constant factor times $\epsilon_0^{-\frac{1}{d-1}}$. For this choice, if we let $\epsilon_0$ be fixed and let $L_0 \to \infty$, we get a contradiction, as $\alpha>1$.
\end{remark}

From here we will copy the arguments of FSS, with modification as appropriate. We will inductively define three sequences $(\epsilon_n)_{n \geq 0}$,
\[
l_1, l_2, \ldots, \quad \text{and} \quad t_1, t_2, \ldots,
\]
and we will set for $k \geq 0$
\[
L_k = L_0 \cdot l_1 \cdot \cdots \cdot l_k, \text{ and } T_k = t_0 + t_1 + \cdots + t_k, \quad \text{with } t_0=0\ .
\]
Next define cubes of scale $k$, $k=0, 1, \ldots$ (pictured in Figure~\ref{fig:bootstrap_boxes}) as
\[
B_k^i = \{0, \ldots, L_k-1\}^d + L_k i \text{ for } i \in \mathbb{Z}^d
\]
and larger cubes
\[
\widetilde B_k^i = \left( \bigcup_{j \in \overline{B}_k} B_{k-1}^j \right) + L_k i, \text{ where } \overline{B}_k = \left\{ - \left\lfloor \frac{1}{3} l_k \right\rfloor , \ldots, l_k + \left\lfloor \frac{1}{3} l_k \right\rfloor \right\}^d
\]
with $B_k := B_k^0$ and $\widetilde B_k := \widetilde B_k^0$.

Just as in FSS, we will run a block dynamics, coupled to the original dynamics, so that the evolution of the spins in each box $B_k^i,~ i \in \mathbb{Z}^d$ during the time interval $[T_{k-1},T_k]$ will depend only on the configuration at time $T_{k-1}$ and the Poisson clocks inside the box $\widetilde B_k^i$. For this we also need the ``influence time'' associated with the box $\widetilde B_k^i$. Let $\left( \sigma_{\widetilde B_k^i,\zeta;s}^{\xi,T_{k-1}}\right)_{s \geq T_{k-1}}$ be the evolution in the box $\widetilde B_k^i$ with boundary condition $\zeta$ outside this box, started at time $T_{k-1}$ from the configuration $\xi$ inside the box. We use $+$ in place of $\zeta$ or $\xi$ to denote the all $+1$ initial configurations, and similarly for $-$. Set
\[
\tau_k^i = \inf \left\{ s \geq T_{k-1}: \sigma_{\widetilde B_k^i,+;s}^{\xi,T_{k-1}}(x) \neq \sigma_{\widetilde B_k^i,-;s}^{\xi,T_{k-1}}(x) \text{ for some } x \in B_k^i \text{ and } \xi \in \{-1,1\}^{\widetilde B_k^i} \right\}\ .
\]

We begin the block process at time $t=t_0$. With probability $1-\epsilon_0$, all spins in the box $B_0^i$ are declared $+1$, and with probability $\epsilon_0$ they are declared $-1$. The status of spins in different boxes is determined independently.

For $k \geq 1$, use the rules of FSS \cite[p.~507]{FSS}:
\begin{enumerate}
\item[]
\begin{enumerate}
\item[{\bf Rule 1}.] During the time interval $[T_{k-1},T_k)$ we observe the evolution inside the box $\widetilde B_k^i$ with $+1$ boundary conditions. We assign to the spins in the box $B_k^i$ up to time $\min\{\tau_k^i,T_k\}$ the values that we see in that evolution.
\item[{\bf Rule 2}.] If $\tau_k^i < T_k$, then at time $\tau_k^i$, all spins in $B_k^i$ are declared to be $-1$, and persist in this state without change to time $T_k$.
\item[{\bf Rule 3}.] If, following the two rules above, there is any spin in state $-1$ in $B_k^i$ at times that are arbitrarily close to $T_k$, then at time $T_k$, all the spins in $B_k^i$ are declared to be $-1$. Otherwise, at time $T_k$ all the spins in $B_k^i$ are declared to be $+1$.
\end{enumerate}
\end{enumerate}

We next note that the following properties \cite[p.~506]{FSS} hold.
\begin{enumerate}
\item[(A)] The block dynamics favors $-1$ spins, in the sense that at any site and time where the original dynamics has a $-1$ spin, the block dynamics also has a $-1$ spin.
\item[(B)] In the block dynamics at time $T_k$, all squares of the $k$-th scale will be monochromatic.
\item[(C)] For each $k \geq 1$, the random field $\eta_k$ that associates to each $i \in \mathbb{Z}^d$ a random variable $\eta_k(i)$ which takes the value $+1$ (respectively $-1$) if at time $T_k$ the block $B_k^i$ is in state $+1$ (respectively $-1$) is a $1$-dependent random field.
\end{enumerate}

For the rest of Section~\ref{sec: FSS_sec}, all probabilities are evaluated on the space on which we have coupled the original dynamics with the block dynamics.

\subsection{Main bounds}

Let $\widetilde \epsilon_k$, $k \geq 0$ denote the probability that at time $T_k$, the block $B_k$ is in state $-1$ and note that $\widetilde \epsilon_0 \leq \epsilon_0$, since $\widetilde \epsilon_0 = \epsilon_0$. We will now try to bound the probability of the event
\[
F_{k+1} = \{-1 \text{ spins are present in }B_{k+1} \text{ at times arbitrarily close to }T_{k+1}\}, ~ k \geq 0
\]
(in steps 1 and 2 below) and of the event
\[
\{\tau_{k+1} < T_{k+1}\},~ k \geq 0
\]
(in step 3 below). Combining these bounds will give us an inequality for $\widetilde \epsilon_{k+1}$ in terms of our various parameters (see the top of Section~\ref{sec: choose}). After that, we will choose appropriate parameters to ensure that we can prove in Section~\ref{sec: induction} that $\widetilde \epsilon_{k+1}$ decreases to zero rapidly.

\bigskip
\noindent
{\bf Step 1. Control of bootstrapping at time }$T_k$. At time $T_k$, all blocks $B_k^i$ are monochromatic and we identify each block in the natural way with an element of $\overline B_{k+1}$. For $i \in \overline B_{k+1}$, let 
\[
\eta_k(i) = \begin{cases}
+1 &\quad \text{if }B_k^i \text{ is in state } +1 \text{ at time } T_k \\
-1 &\quad \text{otherwise}
\end{cases}.
\]
We apply the threshold-two $+1\to -1$ bootstrap percolation rule to the random field $\eta_k$ in $\overline B_{k+1}$ to obtain a collection $\overline R_1, \ldots, \overline R_N$ of well-separated rectangles (no vertex of $\mathbb{Z}^d$ is at distance $\leq 1$ from two rectangles), which is minimal among those that contain the renormalized sites $i \in \overline B_{k+1}$ with $\eta_k(i)=-1$. Precisely, we define a sequence $(\eta_k^{(j)})_{j \geq 0}$ of fields by setting $\eta_k^{(0)} = \eta_k$, and for each $j$, we put $\eta_k^{(j+1)}(i) = -1$ if $i$ has at least two neighbors (in $\ell^1$-distance) in $\overline B_{k+1}$ with $\eta_k^{(j)}$-value equal to $-1$. We set $\eta_k^{(j+1)}(i) = \eta_k^{(j)}(i)$ otherwise. This sequence $(\eta_k^{(j)})$ is monotone in $j$, so we can set $\eta_k^{(\infty)} = \lim_{j \to \infty} \eta_k^{(j)}$. The collection of vertices $i \in \overline B_{k+1}$ with $\eta_k^{(\infty)}(i) = -1$ forms a collection of minimal well-separated rectangles. Let $R_n = \cup_{j \in \overline R_n}B_k^j$ for $n=1, \ldots, N$ and note that the $R_n$'s are also well-separated. By definition, one has $\mathbb{P}_{\epsilon_0,q,L_0}(\eta_k(i)=-1)=\widetilde \epsilon_k$.

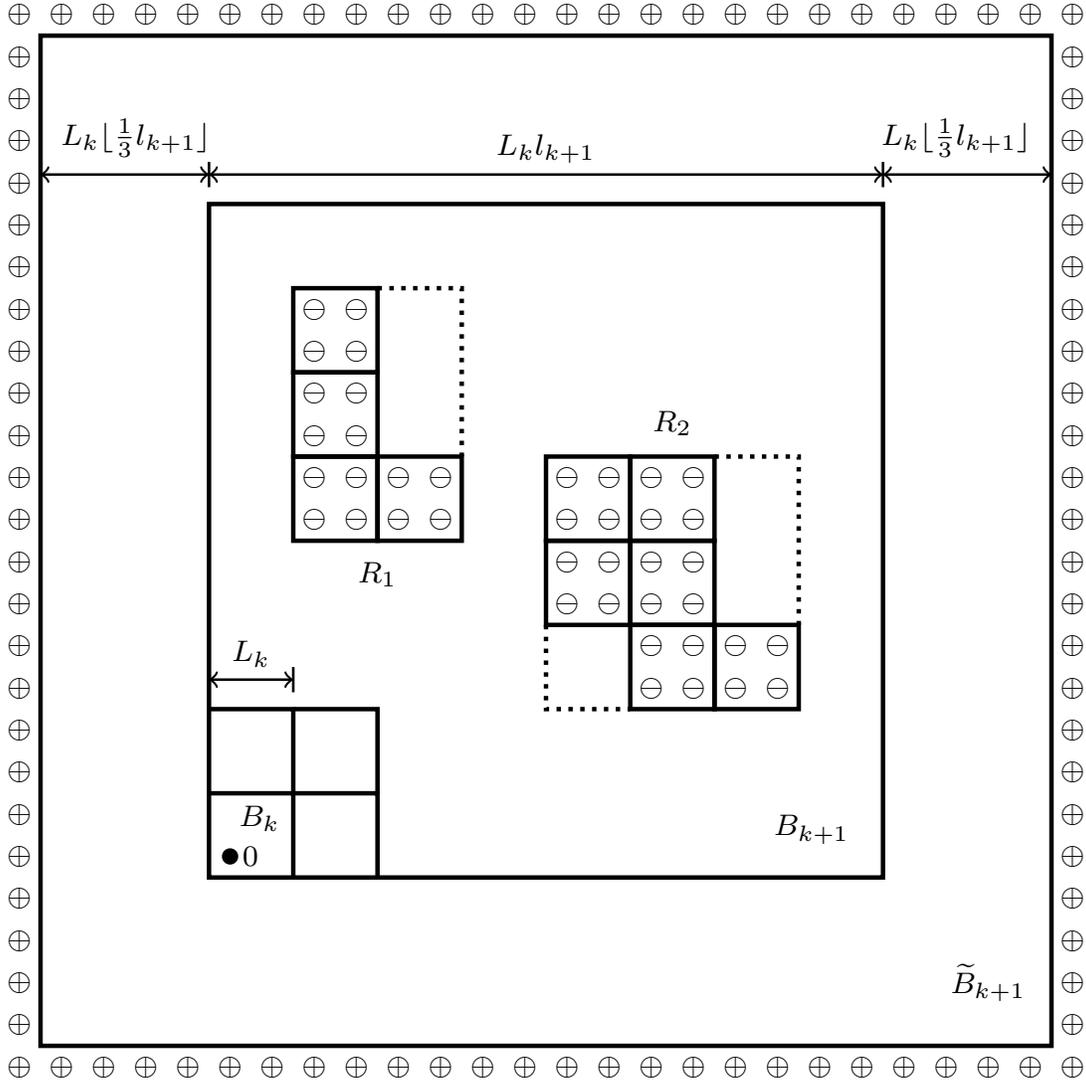
\begin{figure}[htpb]
	\centering
	\scalebox{1.4}{\begin{tikzpicture}[scale=.8, very thick]
		\draw (0,0)--++(8,0)--++(0,8)--++(-8,0)--cycle;
		\draw [fill] (0.25,0.25) circle (2pt);
		\node [right] at  (0.25,0.25) {$0$};
		\node [above left] at  (7.75,0.25) {$B_{k+1}$};
		\draw (0,1)--++(1,0)--++(0,-1);
		\draw (1,1)--++(1,0)--++(0,-1);
		\draw (1,2)--++(1,0)--++(0,-1);
		\draw (0,2)--++(1,0)--++(0,-1);
		\node at (.6,.7) {$B_k$};
		\draw (-2,-2)--++(12,0)--++(0,12)--++(-12,0)--cycle;
		\foreach \ii in {-5,...,20}
			{
				\node at (\ii/2+0.25,-2.25) {$\oplus$};
				\node at (\ii/2+0.25,10.25) {$\oplus$};
			}
		\foreach \jj in {-4,...,19}
			{
				\node at (-2.25,\jj/2+0.25) {$\oplus$};
				\node at (10.25,\jj/2+0.25) {$\oplus$};
			}
		\node at (9.25,-1.25) {$\widetilde B_{k+1}$};
		\draw [line width=.7] (8,8.2)--++(0,.3);
		\draw [line width=.7] (0,8.2)--++(0,.3);
		\draw[<->, line width=.7] (0,8.35)--++(8,0) node [midway, yshift=7]
		{$L_k l_{k+1}$};
		\draw[<->, line width=.7] (8,8.35)--++(2,0) node [midway, yshift=10,
			xshift=-3]
		{$L_k\lfloor \frac{1}{3}l_{k+1} \rfloor$};
		\draw[<->, line width=.7] (-2,8.35)--++(2,0) node [midway,
			yshift=10, xshift=3]
		{$L_k\lfloor \frac{1}{3}l_{k+1} \rfloor$};
		\draw [line width=.7] (1,2.2)--++(0,.3);
		\draw[<->, line width=.7] (0,2.35)--++(1,0) node [midway, yshift=7]
		{$L_k$};
		\draw (1,4)--++(1,0)--++(0,1)--++(-1,0)--cycle;
		\draw (1,5)--++(1,0)--++(0,1)--++(-1,0)--cycle;
		\draw (1,6)--++(1,0)--++(0,1)--++(-1,0)--cycle;
		\draw (2,4)--++(1,0)--++(0,1)--++(-1,0)--cycle;
		\draw[dotted] (2,7)--++(1,0)--++(0,-2);
		\node at (2,3.6) {$R_1$};
		\foreach \ii in {8,...,13}
			{
				\node at (1.25,\ii/2+.25) {$\ominus$};
				\node at (1.75,\ii/2+.25) {$\ominus$};
			}
		\foreach \ii in {8,...,9}
			{
				\node at (2.25,\ii/2+.25) {$\ominus$};
				\node at (2.75,\ii/2+.25) {$\ominus$};
			}
		\foreach \kk in {(4,3),(4,4),(5,2),(5,3),(5,4),(6,2)}
			{
				\draw \kk--++(1,0)--++(0,1)--++(-1,0)--cycle;
			}
		\draw[dotted] (4,3)--++(0,-1)--++(1,0);
		\draw[dotted] (6,5)--++(1,0)--++(0,-2);
		\node at (5.5,5.4) {$R_2$};
		\foreach \ii in {6,...,9}
			{
				\node at (4.25,\ii/2+.25) {$\ominus$};
				\node at (4.75,\ii/2+.25) {$\ominus$};
			}
		\foreach \ii in {4,...,9}
			{
				\node at (5.25,\ii/2+.25) {$\ominus$};
				\node at (5.75,\ii/2+.25) {$\ominus$};
			}
		\foreach \ii in {4,...,5}
			{
				\node at (6.25,\ii/2+.25) {$\ominus$};
				\node at (6.75,\ii/2+.25) {$\ominus$};
			}
	\end{tikzpicture}}
	\caption{Illustration of the hierarchy of cubes used in the proof of Theorem~\ref{thm: block_FSS}. At time $T_k$, the spins in cubes of side length $L_k$ (which are translates of the cube $B_k$ of the form $B_k^j$) are monochromatic. Their boundaries in the dual lattice appear in the figure. Inside the cube $B_{k+1}$, the $R_i$'s are the well-separated rectangles from step 1 which contain the -1 blocks $B_k^j$ and are produced using the $+1\to-1$ bootstrap percolation rule. Throughout the interval of time $[T_k,T_{k+1})$, the evolution is run using $+1$ boundary conditions on $\widetilde B_{k+1}$, the largest box pictured.}
	\label{fig:bootstrap_boxes}
\end{figure}

Estimation of the sizes of the rectangles $\overline R_1, \ldots, \overline R_N$ is similar to that in \cite[p. 508-510]{FSS}. To summarize, one applies the Aizenman-Lebowitz lemma \cite[Lemma~1]{AL} (restated as \cite[Lemma~2.1]{FSS}), to deduce that if one of the rectangles $\overline R_n$ has one side of length bigger than $j$, then $\overline R_n$ must contain a subrectangle $\hat R$ with larger side in $\{\lfloor j/2 \rfloor -1, \ldots, j\}$ which is internally spanned. To bound the probability that there is such an internally spanned subrectangle, we let $n_k$ be any number satisfying
\begin{equation}\label{eq: n_k_condition}
n_k \in \{1, \ldots, \lfloor 5l_{k+1}/3 \rfloor\}.
\end{equation}
(our $n_k$ corresponds to the quantity $\lfloor b/(q_k)^{\frac{1}{d-1}}\rfloor$ in the definition of $E_{k+1}$ below (4.11) in \cite{FSS}) and note that the number of rectangles inside $\overline B_{k+1}$ with the length of the larger side being in $\{\lfloor n_k/2 \rfloor -1, \ldots, n_k\}$ is at most $\left( \frac{5}{3} l_{k+1}\right)^d n_k^d$. Next one can also argue that if $R \subset \overline B_{k+1}$ is an $n_1 \times \cdots \times n_d$ rectangle, then (see \cite[Eq.~(4.15)]{FSS})
\[
\mathbb{P}_{\epsilon_0,q,L_0}(R \text{ is internally spanned}) \leq (2n_1\cdots n_{d-1}\widetilde \epsilon_k)^{\lfloor n_d/3 \rfloor}.
\]
So if we define
\[
E_{k+1} = \{\overline R_1, \ldots, \overline R_N \text{ have all sides of length at most } n_k\},
\]
then assuming \eqref{eq: n_k_condition}, we have
\begin{equation}\label{eq: end_of_bootstrapping}
\mathbb{P}_{\epsilon_0,q,L_0}(E_{k+1}^c) \leq \left( \frac{5n_k l_{k+1}}{3} \right)^d (2n_k^{d-1} \widetilde \epsilon_k)^{\lfloor n_k /3 \rfloor}.
\end{equation}

\bigskip
\noindent
{\bf Step 2. Erosion of $(-1)$-rectangles}. This step follows the corresponding step in \cite[p.~510]{FSS}, with the exception that we allow for general $\alpha>1$ in \eqref{eq: t_k_condition} below. Since we have just given a bound on the probability of $E_{k+1}^c$, and we wish to estimate the probability of $F_{k+1}$, we must next upper bound the conditional probability $\mathbb{P}_{\epsilon_0,q,L_0}(F_{k+1} \mid E_{k+1})$. To do so, we need to consider the system started at time $T_k$ from a configuration in $\widetilde B_{k+1}$ for which $E_{k+1}$ can occur and let the system evolve with $+$ boundary conditions until time $T_{k+1}$. By attractiveness, an upper bound on the probability that in such a setup, there are $-1$ spins present at time $T_{k+1}$ can be obtained by starting the evolution inside $\widetilde B_{k+1}$ at time $T_k$ with $-1$ spins at all sites of the rectangles $R_1, \ldots, R_N$ described in step 1 and $+1$ spins elsewhere in $\widetilde B_{k+1}$. The $-1$ spins cannot spread outside of the rectangles $R_1, \ldots, R_N$, and when a rectangle $R_n$ is taken over by $+1$ spins, it will never again contain a $-1$ spin.

If $E_{k+1}$ occurs, then each $R_n,~ n=1, \ldots, N$, is contained in a cube of side length bounded by $n_k  L_k$. By attractiveness, the time needed to erode such a cube $R$ is therefore stochastically bounded by the time needed to erode a cube with side length $n_k L_k$. We will then make use of Assumption~\ref{eq: assumption_erosion} with $\alpha \geq 1$ to conclude that if $E_{k+1}$ occurs and 
\begin{equation}\label{eq: t_k_condition}
t_{k+1} \geq C (n_k L_k)^\alpha,
\end{equation}
then the probability that at time $T_{k+1} = T_k + t_{k+1}$ there is any $-1$ spin inside a fixed $R_n$ is bounded above by
\[
\exp \left( - \gamma n_k L_k \right).
\]
But the number $N$ of rectangles $R_i$ satisfies
\[
N \leq \left( \frac{5}{3} l_{k+1} \right)^d,
\]
\begin{equation}\label{eq: step_2}
\mathbb{P}_{\epsilon_0,q,L_0}(F_{k+1} \mid E_{k+1}) \leq \left( \frac{5}{3} l_{k+1} \right)^d \exp \left( -\gamma n_k L_k \right).
\end{equation}

\bigskip
\noindent
{\bf Step 3. Control of the outer influence.}  This step directly follows the corresponding step in \cite[p.~511]{FSS}.
Consider the two evolutions $\sigma_{\widetilde B_{k+1},+;t}^{\xi,T_k}$ and $\sigma_{\widetilde B_{k+1},-;t}^{\xi,T_k}$. We say that at time $t \geq T_k$ there is a discrepancy at $x \in \widetilde B_{k+1}$ if there exists some initial $\xi$ such that these evolutions disagree at time $t$ at site $x$. Otherwise $x$ is an agreement vertex. Note that all vertices in $\widetilde B_{k+1}$ are agreement vertices at time $T_k$ and all vertices outside $\widetilde B_{k+1}$ have discrepancies. The time $\tau_{k+1}$ is defined as the first time a vertex in $B_{k+1}$ has a discrepancy.

To estimate the probability that there is a discrepancy in $B_{k+1}$, one first shows that if some vertex $x$ in the internal boundary of $B_{k+1}^i$ at time $t>T_k$ is occupied by a discrepancy, then there exists a \df{chronological path} in the time interval $(T_k,t)$ which starts at some vertex of the external boundary of $\widetilde B_{k+1}^i$ and ends at a vertex of the internal boundary of $B_{k+1}^i$. A chronological path is a self-avoiding path whose vertices are first occupied by discrepancies in order along the path. See \cite[p. 511-512]{FSS} for a proof. Next, one can use the Chernoff bound for Poisson random variables to prove that the probability that any path with $r$ vertices is a chronological path during the time period $(T_k,T_{k+1})$ is at most $e^{-(\log(r/t_{k+1})-1)r}$. From these two facts, we can conclude that
\begin{align}
\mathbb{P}_{\epsilon_0,q,L_0}(\tau_{k+1}\leq T_{k+1}) &\leq \sum_{r \geq \lfloor L_{k+1}/4 \rfloor} 4d L_{k+1} (2d)^r e^{-(\log (r/t_{k+1})-1)r} \nonumber \\
&= 4d L_{k+1} \sum_{r \geq \lfloor L_{k+1}/4 \rfloor} (2d e^{-(\log (r/t_{k+1})-1)})^r\ . \label{eq: step_3}
\end{align}

\subsection{Choosing parameters}\label{sec: choose}
Summarizing, if $q \in [0,1]$, Assumption~\ref{eq: assumption_erosion} holds, and \eqref{eq: n_k_condition} and \eqref{eq: t_k_condition} hold, then we combine \eqref{eq: end_of_bootstrapping}, \eqref{eq: step_2}, and \eqref{eq: step_3} for
\begin{align}
\widetilde \epsilon_{k+1} &\leq \mathbb{P}_{\epsilon_0,q,L_0}(F_{k+1}) + \mathbb{P}_{\epsilon_0,q,L_0}(\tau_{k+1} \leq T_{k+1}) \nonumber \\
&\leq \left( \frac{5n_kl_{k+1}}{3} \right)^d (2n_k^{d-1} \widetilde \epsilon_k)^{\lfloor n_k/3 \rfloor} + \left( \frac{5}{3}l_{k+1}\right)^d \exp \left( - \gamma n_k L_k \right) \label{eq: master_1}\\
&+ 4dL_{k+1} \sum_{r \geq \lfloor L_{k+1}/4 \rfloor} (2de^{-(\log (r/t_{k+1})-1)})^r \label{eq: master_2}\ .
\end{align}

In this section, we will choose all the parameters that appear in the above inequalities, in an effort to minimize $\widetilde \epsilon_{k+1}$ subject to all of our constraints. In the original FSS argument, parameter choices were made along the way, in the three steps above. We postpone our choices so that we can separate the cases $\alpha>1$ and $\alpha=1$ (their paper only involved $\alpha=d$). This will allow us to get better rates depending on the different cases. Here we will split the analysis into two cases, $\alpha=1$ and $\alpha>1$, as we need to choose different parameters in these different cases.

\subsubsection{The case $\alpha=1$} 
Given $q\in [0,1]$, we will select $\epsilon > 0$ below depending only on the dimension $d$ and the constant $C$ from \eqref{eq: assumption_erosion}. Then choose any
\[
\epsilon_0 \in (0,\epsilon),~L_0 \geq 4, \text{ and } t_0 = 0\ .
\]
Next, for positive $\chi,D>0$ to be determined in the next subsection, define for $k \geq 1$,
\[
\epsilon_k = \exp\left( -\frac{\chi}{\epsilon_{k-1}^{\frac{1}{d-1}}}\right),~ l_k = \left\lfloor \frac{D}{\epsilon_{k-1}^{\frac{1}{d-1}}} \right\rfloor,~ t_k = C n_{k-1} L_{k-1}\ ,
\]
where $L_k$ was defined as $L_0 l_1 \cdots l_k$, $C$ is from Assumption~\ref{eq: assumption_erosion} and
\[
n_k = \left\lfloor \frac{1}{(2e \cdot \epsilon_k)^{\frac{1}{d-1}}} \right\rfloor \text{ for } k \geq 0\ .
\]
We will make the choices
\begin{equation}\label{eq: D_choice}
D = \frac{6}{5(2e)^{\frac{1}{d-1}}} + 32d C
\end{equation}
and
\begin{equation}\label{eq: chi_choice}
\chi = \min\left\{ \frac{1}{24 \cdot (2e)^{\frac{1}{d-1}}}, \frac{\gamma}{4\cdot(2e)^{\frac{1}{d-1}}}, \frac{D\log 2}{64} \right\}\ .
\end{equation}

Just as in FSS, we will define
\begin{equation}\label{eq: hat_epsilon_def}
\hat \epsilon = \hat \epsilon(q) = \sup \{x > 0 : \text{ if } \epsilon \in (0,x),~\text{then }\epsilon_k \leq \epsilon \text{ for } k \geq 0\}\ .
\end{equation}
If $\epsilon$ is small enough then $\epsilon_1 \leq \epsilon_0 < \epsilon$ and then, by induction, $\epsilon_k$ is decreasing in $k$. Therefore $\hat \epsilon>0$. We will need to take $\epsilon$ possibly even smaller  than $\hat \epsilon$, so that for each $\epsilon' \in (0,\epsilon] \subset (0,\hat \epsilon]$, the following list of conditions holds. The reader may think of these conditions as requiring that $\epsilon'$ (and therefore $\epsilon_k$) be ``sufficiently small'' at various points in the proof. Inequalities (E1)-(E5) are used below in Section~\ref{sec: induction} for the inductive argument, and (E6)-(E8) are used to establish our main result ``on a subsequence,'' in Section~\ref{sec: subsequence}.
\begin{enumerate}
\item[(E1). ] $\epsilon' \leq (3^{d-1}2e)^{-1}$.
\item[(E2). ] $\epsilon' \leq D^{d-1}$.
\item[(E3). ] 
\[
\exp \left( - \frac{\frac{1}{12 \cdot (2e)^{\frac{1}{d-1}}}-\chi}{(\epsilon')^{\frac{1}{d-1}}} \right) \cdot \left( \frac{5D}{3(2e)^{\frac{1}{d-1}}} \right)^d (\epsilon')^{-\frac{2d}{d-1}} \leq \frac{1}{4}\ .
\]
\item[(E4). ]
\[
(5D/3)^d (\epsilon')^{-\frac{d}{d-1}} \exp \left( -\frac{\frac{\gamma}{2(2e)^{\frac{1}{d-1}}} - \chi}{(\epsilon')^{\frac{1}{d-1}}}\right) \leq \frac{1}{4}.
\]
\item[(E5). ] Setting $\hat C = 8d \sup_{x \geq 1}  x 2^{-x/16}$,
\[
\hat C \exp\left( - \frac{\frac{D\log 2}{32}-\chi}{(\epsilon')^{\frac{1}{d-1}}}\right) \leq \frac{1}{2}\ .
\]
\item[(E6). ]
\[
\frac{\exp\left( - \chi/(\epsilon')^{\frac{1}{d-1}} \right)}{(\epsilon')^2} \leq 1\ .
\]
\item[(E7). ] putting $\iota = \chi/(d-1)$, one has $\epsilon' \leq \min\{ \iota^{-\iota/(d-1)}, \iota^{-\frac{d-1}{\log \iota}}\}$.
\item[(E8). ] $C(k+1) \cdot \frac{D^k}{(2e)^{\frac{1}{d-1}}} \leq \frac{1}{\epsilon_{k-1}^{\frac{1}{d-1}}}$ for $k \geq 1$. This follows from $\epsilon_k \leq \epsilon_0^k$ for all $k$ and taking $\epsilon_0$ small enough, which holds if $\epsilon'$ is small and $\epsilon_0 \leq \epsilon'$.

\end{enumerate}
Note that none of these conditions depends on $L_0$.



Before we move on we should also verify \eqref{eq: n_k_condition} and \eqref{eq: t_k_condition}. The latter holds by definition. For the first, we must show that 
\[
1 \leq \left\lfloor \frac{1}{(2e \cdot \epsilon_k)^{\frac{1}{d-1}}} \right\rfloor \leq \left\lfloor \frac{5}{3} l_{k+1} \right\rfloor\ .
\]
The left inequality holds for $\epsilon \leq (2e)^{-1}$ (see (E1)). By monotonicity of the floor function, the right holds so long as
\[
1 \leq \frac{5}{3} (2e \cdot \epsilon_k)^{\frac{1}{d-1}} \left\lfloor \frac{D}{\epsilon_k^{\frac{1}{d-1}}} \right\rfloor
\]
and this also holds since $\epsilon_k \leq \epsilon\leq D^{d-1}$ (see (E2)) because $D \geq \frac{6}{5(2e)^{\frac{1}{d-1}}}$.

\subsection{Inductive argument}\label{sec: induction}

Now we must show that 
\[
\widetilde \epsilon_k \leq \epsilon_k \text{ for all } k \geq 0\ .
\]
This holds by definition when $k=0$. So assume it is true for some value of $k$; we will show it for $k+1$ by bounding the terms in \eqref{eq: master_1} and \eqref{eq: master_2} one by one. In the argument, we will repeatedly use that
\[
\frac{1}{2 \cdot (2e\epsilon_k)^{\frac{1}{d-1}}} \leq n_k \leq \frac{1}{(2e \epsilon_k)^{\frac{1}{d-1}}}\ ,
\]
which holds since $\epsilon_k \leq (2e)^{-1}$ (see (E1)), and 
\[
\frac{D}{2\epsilon_k^{\frac{1}{d-1}}} \leq l_{k+1} \leq \frac{D}{\epsilon_k^{\frac{1}{d-1}}}\ ,
\]
which holds since $\epsilon_k \leq \epsilon \leq D^{d-1}$ (see (E2)).

For the first term of \eqref{eq: master_1}, use the inductive hypothesis that $\widetilde \epsilon_k \leq \epsilon_k$:
\[
(2n_k^{d-1} \widetilde \epsilon_k)^{\lfloor n_k/3\rfloor} = \exp\left( \lfloor n_k/3 \rfloor \log (2n_k^{d-1}\widetilde \epsilon_k) \right) \leq \exp \left( - \lfloor n_k/3 \rfloor \right)\ .
\]
As long as $n_k \geq 3$ (see (E1)), an upper bound is $e^{-n_k/6}$ and we obtain
\[
(2n_k^{d-1} \widetilde \epsilon_k)^{\lfloor n_k/3 \rfloor} \leq \exp \left( - \frac{1}{12\cdot (2e \epsilon_k)^{\frac{1}{d-1}}} \right) = \epsilon_{k+1} \exp \left( - \frac{\frac{1}{12 \cdot (2e)^{\frac{1}{d-1}}}-\chi}{\epsilon_k^{\frac{1}{d-1}}} \right)\ .
\]
The next term of \eqref{eq: master_1} is bounded as
\[
\left( \frac{5n_kl_{k+1}}{3} \right)^d \leq \left( \frac{5D}{3(2e)^{\frac{1}{d-1}}} \right)^d \epsilon_k^{-\frac{2d}{d-1}}\ .
\]
Therefore, along with (E3),
\begin{equation}\label{eq: term_1}
\left( \frac{5n_kl_{k+1}}{3} \right)^d (2n_k^{d-1}\widetilde \epsilon_k)^{\lfloor n_k/3\rfloor} \leq \epsilon_{k+1} \left[ \exp \left( - \frac{\frac{1}{12 \cdot (2e)^{\frac{1}{d-1}}}-\chi}{\epsilon_k^{\frac{1}{d-1}}} \right) \cdot \left( \frac{5D}{3(2e)^{\frac{1}{d-1}}} \right)^d \epsilon_k^{-\frac{2d}{d-1}} \right] \leq \epsilon_{k+1}/4\ .
\end{equation}

We move to the rightmost term of \eqref{eq: master_1}: 
\begin{align*}
\left( \frac{5}{3} l_{k+1} \right)^d \exp \left( - \gamma n_k L_k \right) &\leq \left(\frac{5D}{3}\right)^d \epsilon_k^{-\frac{d}{d-1}} \exp\left(-\frac{\gamma}{2(2e\epsilon_k)^{\frac{1}{d-1}}}\right) \\
&= \epsilon_{k+1} (5D/3)^d \epsilon_k^{-\frac{d}{d-1}} \exp \left( -\frac{\frac{\gamma}{2(2e)^{\frac{1}{d-1}}} - \chi}{\epsilon_k^{\frac{1}{d-1}}}\right)\ .
\end{align*}
By (E4),
\begin{equation}\label{eq: term_2}
\left( \frac{5}{3} l_{k+1} \right)^d \exp \left( - \gamma n_k L_k \right) \leq \epsilon_{k+1}/4\ .
\end{equation}

Finally we bound the last term of \eqref{eq: master_1}; recall it is
\[
4dL_{k+1} \sum_{r \geq \lfloor L_{k+1}/4 \rfloor} (2de^{-(\log (r/t_{k+1})-1)})^r\ .
\]
For the sum, consider an integer $x$ and 
\[
\sum_{r \geq x} (2de^{-(\log(r/t_{k+1})-1)})^r = \sum_{r \geq x} \left( 2de \frac{t_{k+1}}{r} \right)^r\ .
\]
If $x \geq 4det_{k+1}$ then this no bigger than
\[
\sum_{r \geq x} 2^{-r} = 2^{-x+1}\ .
\]

We would like to use $x=\lfloor L_{k+1}/4 \rfloor$, so we must verify that 
\begin{equation}\label{eq: time_upper_bound}
\lfloor L_{k+1}/4 \rfloor \geq 4de \cdot t_{k+1}.
\end{equation} 
Note that since we have taken $L_0 \geq 4$ and (E2) implies that $l_k \geq 1$ for all $k$, we also have $L_{k+1}/4 \geq 1$, so
\[
\lfloor L_{k+1}/4 \rfloor \geq  L_k l_{k+1}/8 \geq L_k \frac{D}{16\epsilon_k^{\frac{1}{d-1}}} = Cn_kL_k \frac{D}{16Cn_k \epsilon_k^{\frac{1}{d-1}}} \geq Cn_kL_k \frac{De}{8C}\ .
\]
By choice of $D$, one has $D \geq 32dC$, so we obtain $\lfloor L_{k+1}/4\rfloor \geq 4deCn_kL_k = 4de \cdot t_{k+1}$. Therefore, using $L_0 \geq 4$,
\[
4dL_{k+1} \sum_{r \geq \lfloor L_{k+1}/4 \rfloor} (2de^{-(\log (r/t_{k+1})-1)})^r \leq 8dL_{k+1}2^{-\lfloor L_{k+1}/4\rfloor} \leq 8dL_{k+1}2^{-L_{k+1}/8}\ .
\]
By $L_{k+1} \geq l_{k+1} \geq \frac{D}{2\epsilon_k^{\frac{1}{d-1}}}$, our upper bound becomes
\[
8dL_{k+1} 2^{-L_{k+1}/16} 2^{-\frac{D}{32 \epsilon_k^{\frac{1}{d-1}}}}\ .
\]
This is bounded using (E5) as
\begin{equation}\label{eq: term_3}
\epsilon_{k+1} \cdot 8dL_{k+1} 2^{-L_{k+1}/16}\exp\left( - \frac{\frac{D\log 2}{32}-\chi}{\epsilon_k^{\frac{1}{d-1}}}\right) \leq \epsilon_{k+1}/2\ .
\end{equation}

\subsubsection{The case $\alpha>1$}

In the case $\alpha>1$, one can make the following choice of parameters: for some $\epsilon>0$ small enough,
\begin{equation}\label{eq: initial}
\epsilon_0 \in (0,\epsilon),~ L_0 \leq K\epsilon_0^{-\frac{2\delta}{(d-1)(\alpha-1)}},~t_0=0,
\end{equation}
where $K = \frac{(2e)^{\frac{\alpha}{(\alpha-1)(d-1)}}}{(32de)^{\frac{1}{\alpha-1}}}$. Then we put for $k=1, 2, \ldots$,
\[
\epsilon_{k+1} = \exp\left( - \frac{\chi}{\epsilon_k^{\frac{1}{d-1}}} \right),~ l_{k+1} = \left\lfloor \frac{1}{\epsilon_k^{\frac{\alpha+2\delta}{d-1}}} \right\rfloor,~ t_{k+1} = C(n_kL_k)^\alpha,
\]
where $\chi$ is chosen small enough, and for $k=0, 1, \ldots$,
\[
n_k = \left\lfloor \frac{1}{(2e\epsilon_k)^{\frac{1}{d-1}}} \right\rfloor.
\]
Similarly to the last section, one can then show that $\widetilde\epsilon_k \leq \epsilon_k$ for all $k = 0, 1, 2, \ldots$.


\subsection{Final result}

\subsubsection{On a subsequence: the case $\alpha=1$} \label{sec: subsequence}
Combining \eqref{eq: term_1}, \eqref{eq: term_2} and \eqref{eq: term_3}, we obtain $\widetilde \epsilon_{k+1} \leq \epsilon_{k+1}$. So by induction, $\widetilde \epsilon_k \leq \epsilon_k$ for all $k \geq 0$, and by attractiveness,
\begin{equation}\label{eq: before_gaps}
\mathbb{P}_{\epsilon_0,q,L_0}(\sigma_v^t = +1 \text{ for all } v \in \Lambda_{L_k} \text{ and } t = T_k) \geq 1-\epsilon_k\ .
\end{equation}
To turn this into a bound involving not $1-\epsilon_k$, but instead a function of $T_k$, we first note
\begin{align*}
t_{k+1} = Cn_kL_k &\leq C \cdot \frac{1}{(2e\epsilon_k)^{\frac{1}{d-1}}}L_0 l_1 \cdots l_k \\
&\leq C \cdot \frac{D^k}{(2e)^{\frac{1}{d-1}}} \cdot \frac{L_0}{(\epsilon_k \cdots \epsilon_0)^{\frac{1}{d-1}}} \text{ for } k \geq 0\ .
\end{align*}
We claim that for all $k \geq 0$,
\begin{equation}\label{eq: epsilon_bound}
\frac{1}{\epsilon_{k-2} \cdots \epsilon_0} \leq \frac{1}{\epsilon_{k-1}}\ .
\end{equation}
Here, we interpret $\epsilon_{-\ell} = 1$ for $\ell \geq 1$. For $k\leq 2$ it is true due to monotonicity of $\epsilon_k$ in $k$. Assuming it holds for some value of $k$, we bound using (E6):
\[
\frac{1}{\epsilon_{k-1} \cdots \epsilon_0} \leq \frac{1}{\epsilon_{k-1}^2} = \frac{1}{\epsilon_k} \frac{\exp\left( - \chi/\epsilon_{k-1}^{\frac{1}{d-1}}\right)}{\epsilon_{k-1}^2} \leq \frac{1}{\epsilon_k}\ .
\]

Therefore the bound we give for $t_{k+1}$ is
\[
t_{k+1} \leq C \frac{D^k}{(2e)^{\frac{1}{d-1}}} \frac{L_0}{(\epsilon_k \epsilon_{k-1}^2)^{\frac{1}{d-1}}} \text{ for } k \geq 0
\]
and by monotonicity and (E8), putting $\iota = \exp\left( \chi/(d-1) \right)$,
\begin{align}
T_{k+1} &\leq C(k+1) \cdot \frac{D^k}{(2e)^{\frac{1}{d-1}}}  \frac{L_0}{(\epsilon_k \epsilon_{k-1}^2)^{\frac{1}{d-1}}} \nonumber \\
&= C(k+1) \cdot \frac{L_0D^k}{(2e)^{\frac{1}{d-1}}} \log_{\iota}(1/\epsilon_{k+1}^{\frac{1}{d-1}}) \left( \log_{\iota} \log_{\iota}(1/\epsilon_{k+1}^{\frac{1}{d-1}})\right)^2 \nonumber \\
&\leq L_0 \log_{\iota} (1/\epsilon_{k+1}^{\frac{1}{d-1}}) \left( \log_{\iota} \log_{\iota}(1/\epsilon_{k+1}^{\frac{1}{d-1}}) \right)^3.  \label{eq: t_k_bound}
\end{align}
For $x \geq \iota$, setting $y=L_0 x(\log_{\iota} x)^3$,
\[
\frac{L_0 x}{y/(\log_{\iota} y)^3} = \left( \frac{\log_{\iota} (L_0x) + 3 \log_{\iota} \log_{\iota} x}{\log_{\iota} x} \right)^{3} \geq 1\ ,
\]
so for such $x$, $x \geq \frac{y}{L_0 \log_{\iota}^3 y}$. Applying this in \eqref{eq: t_k_bound} with $x=\log_{\iota}\left(1/\epsilon_{k+1}^{\frac{1}{d-1}}\right)$, noting (E7), and using the fact that $y \mapsto \frac{y}{L_0\log_{\iota}^3 y}$ is monotone for $y \geq \iota^{\frac{3}{\log \iota}}$ (which itself guaranteed for our choice of $y$ by (E7)),
\[
\log_{\iota}\left( \frac{1}{\epsilon_{k+1}^{\frac{1}{d-1}}}\right) \geq \frac{T_{k+1}}{L_0 \log_{\iota}^3 T_{k+1}} \text{ for }k \geq 0\ ,
\]
or
\[
\epsilon_{k+1} \leq \exp \left( - \frac{\chi^4}{(d-1)^3} \frac{T_{k+1}}{L_0 \log^3 T_{k+1}} \right) \text{ for } k \geq 0\ .
\]
Rewriting \eqref{eq: before_gaps}, we obtain
\[
\mathbb{P}_{\epsilon_0,q,L_0}\left( \sigma_v^t  = +1 \text{ for all } v \in \Lambda_{L_k} \right) \geq 1-\exp\left( -\frac{\chi^4}{(d-1)^3} \frac{t}{L_0 \log^3 t} \right) \text{ for } t=T_k \text{ and } k \geq 0\ .
\]
Note that we could replace $\log^3 t$ by $\log t ~(\log \log)^3t$ or more iterated logs, and by doing this, we can adjust the constants to replace $\log^3 t$ by $\log^2 t$. Thus, for some $M >0$,
\begin{equation}\label{eq: near_end}
\mathbb{P}_{\epsilon_0,q,L_0}\left( \sigma_v^t = +1 \text{ for all } v \in \Lambda_{L_k} \right) \geq 1-\exp\left( -M \frac{t}{L_0 \log^2 t} \right) \text{ for } t=T_k \text{ and } k \geq 0\ .
\end{equation}

\subsubsection{On a subsequence: the case $\alpha>1$}

In this case, we repeat almost the same computations as above, but this time using the values of parameters chosen for $\alpha>1$. We then obtain, for $p = 2\alpha + 5\delta$,
\[
\mathbb{P}_{\epsilon_0,q,L_0}\left( \sigma_v^t  = +1 \text{ for all } v \in \Lambda_{L_k} \right) \geq 1-\exp\left( -\frac{\chi^{p+1}}{(d-1)^p} \frac{t^{1/\alpha}}{L_0 \log^p t} \right) \text{ for } t=T_k \text{ and } k \geq 0.
\]
As before, in this derivation, we can replace $\log^p t$ by $\log^{\alpha+2\delta} t \cdot (\log \log)^p t$, or more iterated logs, and by doing this, we can adjust the constants to replace $\log^pt$ by $\log^{\alpha + 3\delta}t$ for any $\eta>0$. Therefore we can achieve for some $M>0$
\begin{equation}\label{eq: near_end_alpha_not_one}
\mathbb{P}_{\epsilon_0,q,L_0}\left(\sigma_v^t  = +1 \text{ for all } v \in \Lambda_{L_k} \right) \geq 1-\exp\left( -M \frac{t^{1/\alpha}}{\log^{\alpha +3\delta} t}\right) \text{ for } t=T_k \text{ and } k \geq 0\ .
\end{equation}

\subsubsection{Filling in the gaps}

To extend \eqref{eq: near_end} and \eqref{eq: near_end_alpha_not_one} to all $t \geq 0$, we will be comparing evolutions started from configurations sampled from product measures with different values of $\epsilon$. Recall that we have derived \eqref{eq: near_end} under the assumptions that $q \in [0,1]$, \eqref{eq: assumption_erosion} holds for $\alpha = 1$ and some $C$, and, for some $\epsilon = \epsilon(\alpha,C,d) >0$ satisfying (E1)-(E7), one has $\epsilon_0 \in (0,\epsilon]$ with $L_0 \geq 4$. Similar assumptions were made in the case $\alpha>1$ to derive \eqref{eq: near_end_alpha_not_one}, with now $\delta>0$ also given, and $\epsilon = \epsilon(\alpha,C,d,\delta)$, with $L_0$ satisfying the bound in \eqref{eq: initial}. So we will now fix $q, \alpha, d, \delta,C,$ and $L_0$, and, given $t>0$, try to modify $\epsilon_0$ to force $t$ to equal some $T_k$.

For this purpose, if $\eta \in (0,1)$, we write $\epsilon_k(\eta)$, $t_k(\eta)$ and $T_k(\eta)$ for the corresponding values of $\epsilon_k$, $t_k$ and $T_k$ with $\epsilon_0=\eta$. 
We will write $\epsilon_k' = \epsilon_k(\epsilon)$, $t_k' = t_k(\epsilon)$ and $T_k' = T_k(\epsilon)$. Since $\epsilon \in (0,\hat \epsilon]$, $\epsilon_k'$ decreases with $k$ and therefore $t_k'$ increases with $k$. (Recall that $\hat \epsilon$ was defined in \eqref{eq: hat_epsilon_def}.)

For each fixed $k \geq 1$, if we continuously decrease $\eta$ from $\epsilon$ to $\epsilon_1'$, then $T_k(\eta)$ increases continuously from $T_k(\epsilon) = T_k'$ to
\[
T_k(\epsilon_1') = t_1(\epsilon_1') + \cdots + t_k(\epsilon_1') = t_2(\epsilon) + \cdots + t_{k+1}(\epsilon) = T_{k+1}(\epsilon) - t_1(\epsilon) = T_{k+1}'-t_1'\ .
\]
Thus any $t>0$ which is not in $\cup_{k \geq 2} [T_k' - t_1',T_k')$ is of the form $t=T_{k(t)}(\epsilon(t))$ for some $k(t) \geq 1$ and some $\epsilon = \epsilon(t) \in (\epsilon_1', \epsilon]$. Then for any $\epsilon_0 \in (0,\epsilon_1')$ we have $1-\epsilon_0 \geq 1-\epsilon_1' \geq 1-\epsilon(t)$. Therefore, by attractiveness and \eqref{eq: near_end}, we have for $\alpha=1$,
\begin{align*}
\mathbb{P}_{\epsilon_0,q,L_0}(\sigma_0^t  = -1) &\leq \mathbb{P}_{\epsilon(t),q,L_0}(\sigma_0^t=-1) \\
&\leq \exp\left( -M \frac{t}{L_0 \log^2 t} \right) \text{ for } t \notin \cup_{k \geq 2} [T_k'-t_1',T_k')
\end{align*}
and a similar statement for $\alpha>1$.

To extend the result to $t \in \cup_{k \geq 2} [T_k'-t_1',T_k')$, observe that for each $k$ and $t \in [T_k'-t_1',T_k')$, if $\sigma_0^t=-1$ and the spin at the origin does not flip between times $t$ and $T_k'$, then $\sigma_0^{T_k'} =-1$. Using the Markov property, we obtain then for $\alpha=1$,
\begin{align*}
\mathbb{P}_{\epsilon_0,q,L_0}( \sigma_0^t =-1) &\leq e^{t_1'} \mathbb{P}_{\epsilon_0,q,L_0}(\sigma_0^{T_k'}=-1) \\
&\leq e^{t_1'} \exp\left( -M \frac{T_k'}{L_0 \log^2 T_k'} \right),
\end{align*}
and a similar statement for $\alpha>1$, where $e^{-t_1'}$ comes from the probability that no flips occur at the origin from time $t$ to time $T_k'$. Since $t_1'$ is a constant relative to $k$, this shows the bounds of Theorem~\ref{thm: block_FSS} with the event $\{\sigma_0^s = -1 \text{ for some } s \geq t\}$ replaced by the event $\{\sigma_0^t = -1\}$. The rest of the proof from this point (showing the bound for this first event) is identical to that in \cite[p. 514]{FSS} (applying the strong Markov property once), so we omit the details.

\section{Near-exponential fixation}\label{sec: finish}

In this section, we combine Proposition~\ref{q=1 thm}, and Theorems~\ref{exponential box decay} and~\ref{thm: block_FSS} to prove Theorem~\ref{fixation times}.

\begin{proof}
First suppose $d=2$. Theorem~\ref{exponential box decay} implies that if $q>1/2$, then Assumption~\ref{assumption} holds with $\alpha=1$. In this case, part $1$ of Theorem~\ref{thm: block_FSS} applies, so let $\eps>0$ be as in part $1$ of Theorem~\ref{thm: block_FSS}.  By Proposition~\ref{q=1 thm}, we can choose $L_0\ge 4$ sufficiently large (depending on $p$) such that
\begin{equation}\label{box q=1 limit}
\probsub{\lim_{t\to\infty} \sigma^t_x = +1 \text{ for all } x\in \Lambda_{L_0} \, \middle|\, \sigma^0_x = -1 \text{ for all } x\in \Lambda_{L_0}^c}{p,1} > 1-\eps/2.
\end{equation}
By Fatou's lemma,
$$
\liminf_{t\to\infty} \probsub{\sigma^t_x = +1 \text{ for all } x\in \Lambda_{L_0} \, \middle|\, \sigma^0_x = -1 \text{ for all } x\in \Lambda_{L_0}^c}{p,1} > 1-\eps/2,
$$
so we can choose $t_0 = t_0(\eps,L_0,p)$ such that
\begin{equation}
\label{box q=1}
\probsub{\sigma^{t_0}_x = +1 \text{ for all } x\in \Lambda_{L_0} \, \middle|\, \sigma^0_x = -1 \text{ for all } x\in \Lambda_{L_0}^c}{p,1} \geq 1-\eps/2.
\end{equation}
Since $L_0$ and $t_0$ are fixed and finite, the probability in (\ref{box q=1}) varies continuously with $q$.  For example, to show continuity at $q=1$, observe that by decreasing $q$ from $q=1$ to $q=q^*<1$, we may introduce at most a Poisson$(t_0{L_0}^2(1-q^*))$ number of energy-neutral flips from $+1$ to $-1$ in $\Lambda_{L_0}$ by time $t_0$, and we may remove at most a Poisson$(t_0{L_0}^2(1-q^*))$ number of energy-neutral flips from $-1$ to $+1$.  If both of these are zero, then the configurations at time $t_0$ when $q=1$ and $q=q^*$ are identical (by coupling all other flips so they are the same). The probability that the number of different flips is zero can be made larger than $1-\epsilon/2$ by choosing $q^* = q^*(\eps,L_0,t_0,p)<1$ sufficiently close to $1$, so
\begin{equation}
\label{box q*}
\probsub{\sigma^{t_0}_x = +1 \text{ for all } x\in \Lambda_{L_0} \, \middle|\, \sigma^0_x = -1 \text{ for all } x\in \Lambda_{L_0}^c}{p,q^*} =:1-\eps_0 > 1-\eps.
\end{equation}
Now, independently for each $x\in \Z^2$, we run the Glauber dynamics with $q=q^*$ and initial density $p$ of $+1$'s in $xL_0+\Lambda_{L_0}$ and all $-1$'s outside $xL_0+\Lambda_{L_0}$ until time $t_0$. Call the box $xL_0+\Lambda_{L_0}$ a $+$-box if and only if all spins in $xL_0+\Lambda_{L_0}$ are $+1$ at time $t_0$ under these dynamics, and observe that the event that $xL_0+\Lambda_{L_0}$ is a $+$-box depends only on the initial configuration and the sequence of clock rings and tie-breaking coin flips within $xL_0+\Lambda_{L_0}$ up to time $t_0$. Since the initial configurations and sequences of clock rings and coin flips are independent between boxes, it follows that each box is independently a $+$-box. Finally, for each $y\in \Z^2$, we declare $\tilde \sigma^{0}_y = +1$ if $y$ is in a $+$-box, and $\tilde \sigma^{0}_y = -1$ otherwise, and let $(\tilde\sigma^t)_{t\ge0}$ evolve according to the Glauber dynamics with $q=q^*$. If $\sigma^{t_0}$ is the state of the Glauber dynamics with $q=q^*$ and initial density $p$ of $+1$'s, then by attractiveness it follows that $\sigma^{t_0}$ dominates (has more $+1$'s than) $\tilde\sigma^{0}$. By~\eqref{box q*}, the configuration $\tilde\sigma^{0}$ has the property that each box $xL_0+\Lambda_{L_0}$ for $x\in \Z^2$ is filled with $+1$ spins independently with probability $1-\eps_0 > 1-\eps$, and filled with $-1$ spins otherwise. Therefore, $\tilde\sigma^0$ satisfies the conditions of part $1$ of Theorem~\ref{thm: block_FSS}, and we have for $q=q^*$ and $t>t_0$
\begin{align*}
\mathbb{P}_{p,q}(\sigma_0^s = -1 \text{ for some }s \geq t) &\leq \mathbb{P}_{p,q}(\tilde\sigma_0^s = -1 \text{ for some }s \geq t-t_0)\\
&= \mathbb{P}_{\epsilon_0,q,L_0}(\sigma^s_0 = -1 \text{ for some } s \geq t-t_0) \\
&\leq \exp\left( -C_1 \frac{t}{\log^2 t} \right)
\end{align*}
for all large enough $t$. This completes the proof in the case $d=2$.

Now suppose $d\ge 3$ and $\beta>\min(d-1,3)$ are fixed. Let $\alpha\in (\min(d-1,3),\beta)$, and let $C >0$ be the constant of Theorem~\ref{exponential box decay}. For $d=3$, then Assumption~\ref{assumption} is satisfied for this $\alpha>2$ and $C$, and $\gamma=1/C$. For $d\geq 4$, if $L \ge (\log L)^{C/(\alpha-3)}$, then Assumption~\ref{assumption} is satisfied for large enough $L$ for this $\alpha>3$ and $C$, and $\gamma=1/C$. By increasing $C$ further, the assumption is seen to hold for all $L$.

Since $\alpha>1$, we will apply part $2$ of Theorem~\ref{thm: block_FSS}, so we must check that we can choose $\eps_0 \in (0,\eps)$ and $L_0$ such that
\[
(\log L_0)^{C'/(\alpha-3)}\le L_0 \leq K \epsilon_0^{-\frac{2\delta}{(d-1)(\alpha-1)}}.
\]
We arbitrarily choose $\delta = 1/2$. Let $\eps>0$ be as in part $2$ of Theorem~\ref{thm: block_FSS}, and select $\eps' \in (0,\eps)$, which will be specified shortly. Letting $c = c(p,d)$ be the constant in Proposition~\ref{q=1 thm}, we may choose $L_0 = \ceil{\frac{1}{c} \log(2/\eps')} + 1$ to obtain
\begin{equation}\label{box q=1 limit 2}
\probsub{\lim_{t\to\infty} \sigma^t_x = +1 \text{ for all } x\in \Lambda_{L_0} \, \middle|\, \sigma^0_x = -1 \text{ for all } x\in \Lambda_{L_0}^c}{p,1} \ge 1-e^{-c{L_0}} > 1-\eps'/2.
\end{equation}
By choosing $\eps'$ sufficiently small, it follows that for any $\eps_0\in (0,\eps')$ we have
$$
(\log L_0)^{C'/(\alpha-3)} \le L_0 \le  K (\epsilon')^{-\frac{1}{(d-1)(\alpha-1)}} \le  K \epsilon_0^{-\frac{1}{(d-1)(\alpha-1)}}.
$$
The remainder of the proof proceeds in the same way as for $d=2$, but with~\eqref{box q=1 limit 2} in place of~\eqref{box q=1 limit}, $\eps'$ in place of $\eps$, and $d$ in place of $2$ where appropriate.  In this way, part $2$ of Theorem~\ref{thm: block_FSS} gives us a probability bound of $\exp[-C_2 t^{1/\alpha} / (\log t)^{\alpha+3\delta}]$, which is smaller than $\exp[-t^{1/\beta}]$ for large enough~$t$.
\end{proof}

\section*{Acknowledgments} MD thanks Eric Vigoda and Antonio Blanca for discussions related to the erosion times of boxes. MD and DS are grateful to Rob Morris for lengthy email discussions of coarsening dynamics. LP thanks Timo Sepp\"al\"ainen and Benedek Valk\'o for helpful discussions on large deviation estimates in particle systems, and Ivan Corwin for useful remarks. MD is supported by an NSF CAREER grant. LP is partially supported by NSF grant DMS-1664617. DS is supported by NSF grants DMS-1418265 and CCF-1740761.
We are grateful to anonymous referees for valuable suggestions.


\bigskip

\textsc{Michael Damron, Department of Mathematics, Georgia Institute of Technology, Atlanta, GA 30332}

\textit{E-mail address:} \texttt{mdamron6@gatech.edu}

\medskip

\textsc{L. Petrov, Department of Mathematics, University of Virginia, Charlottesville, VA 22904, and Institute for Information Transmission Problems,
Moscow, Russia 127051}

\textit{E-mail address:} \texttt{lenia.petrov@gmail.com}

\medskip

\textsc{David Sivakoff, Departments of Statistics and Mathematics, The Ohio State University, Columbus, OH 43210}

\textit{E-mail address:} \texttt{dsivakoff@stat.osu.edu}

\end{document}